\documentclass{amsart}
\usepackage{graphicx}
\usepackage{subfigure}

\newtheorem{thm}{Theorem}[section]

\newtheorem{lem}[thm]{Lemma}

\newtheorem{rem}[thm]{Remark}

\newtheorem{acknowledgement}{Acknowledgement}

\newcommand{\norm}[1]{\left\Vert#1\right\Vert}
\newcommand{\abs}[1]{\left\vert#1\right\vert}

\newcommand{\Real}{\mathbb R}

\newcommand{\pfrac}[2]{\frac{\partial #1}{\partial #2}}
\begin{document}
\title{On the  Sacks-Uhlenbeck flow of Riemannian surfaces}

\numberwithin{equation}{section}
\author{Min-Chun Hong and Hao Yin}

\address{Min-Chun Hong, Department of Mathematics, The University of Queensland\\
Brisbane, QLD 4072, Australia}  \email{hong@maths.uq.edu.au}

\address{Hao Yin, Department of Mathematics, The University of Queensland\\
Brisbane, QLD 4072, Australia and Department of Mathematics, Shanghai Jiaotong University, Shanghai, China}
\email{haoyin@sjtu.edu.cn}

\begin{abstract}
In this paper, we study an $\alpha$-flow for the Sacks-Uhlenbeck
functional on Riemannian surfaces and prove that the limiting map
by the $\alpha$-flows is a weak solution to the harmonic map flow.
By an application of the $\alpha$-flow, we present a simple proof
of an energy identity of a minimizing sequence in each homotopy
class.
\end{abstract}

 \maketitle

\pagestyle{myheadings} \markright {} \markleft { }

\section{Introduction}

Suppose that $M$ is a Riemannian manifold and $N$ is a closed
manifold embedded in $\Real^k$. A critical point $u$ of the
Dirichlet energy
\begin{equation*}
    E(u)=\int_M \abs{\nabla u}^2 dv
\end{equation*}
is called a harmonic map.

Harmonic maps  between Riemannian manifolds can be thought of as a
natural generalization of geodesics,   minimal surfaces  and
harmonic functions. A fundamental question is: Given a smooth map
$u_0$ from $M$ to $N$, does there exist  a  smooth harmonic map
representative in the homotopy class of $u_0$? In a pioneering
work \cite{ES}, Eells and Sampson introduced the harmonic map flow
and used it to deform an initial map to a harmonic map in the same
homotopic class if the sectional curvature of $N$ is non-positive.
In general, one cannot expect to have the existence of smooth
harmonic maps into general target manifolds due to the fact that
singularities do occur.  From now on, we assume that $M$ is a
closed Riemannian surface, hence $E(u)$ is conformally invariant.
Under certain topological conditions of $N$, the existence of
minimizing harmonic maps in a homotopy class was proved by Lemaire
\cite{Lemaire} and Schoen-Yau \cite{SY}. In a well-known paper
\cite{SU}, Sacks and Uhlenbeck proposed a family of the perturbed
functional
\begin{equation*}
    E_\alpha(u)=\int_M (1+\abs{\nabla u}^2)^\alpha dv
\end{equation*}
for $\alpha>1$. The advantage of the perturbed functional is that
$E_\alpha$ satisfies the Palais-Smale condition and therefore it
is easy to obtain critical points of $E_\alpha$ by either
minimizing energy functional or Morse theory. When $\alpha\to 1$,
the limiting map of critical points of $E_\alpha$ is a harmonic
map and a bubbling phenomenon occurs.

 On the other hand, Struwe
\cite{struwe} proved the global existence of the weak solution  to
the harmonic map flow and that the solution to the flow converges
to a harmonic map as $t\to\infty$. Chang, Ding and Ye \cite{CDY}
constructed an example that the harmonic map flow blowups at
finite time, so in general the limit harmonic map by the flow  may
not be in the same homotopy class of the initial map.

In this paper, we study an $\alpha$-flow for the perturbed energy
$E_\alpha$ in the same homotopy class of the initial map. More
precisely, we consider the following evolution problem:
\begin{equation}\label{eqn:flow1}
    \partial_t u=\triangle_M u+(\alpha-1)\frac{\nabla \abs{\nabla u}^2 \cdot \nabla u}{1+\abs{\nabla u}^2} +A(u)(\nabla u,\nabla u)
\end{equation}
 with $u(x,0)=u_0$, where $\triangle_M$ is the Laplacian operator with respect to the Riemannian metric of $M$ and $A$ is the second fundamental form of $N$.
 We call (1.1) the Sacks-Uhlenbeck flow (or $\alpha$-flow). We would like to point out that
 the $\alpha$-flow is not the standard gradient flow, but the flow has some
analytic advantage.

 At first, we have

\begin{thm}\label{thm:global}
For a given smooth map $u_0:M\to N$, there exists a unique global
smooth solution $u_{\alpha }(x,t)$ to the evolution problem
(\ref{eqn:flow1}) in $M\times [0,\infty )$. Moreover, for any
$t_i\to \infty$, $u_{\alpha}(\cdot,t_i)$ converges smoothly to a
limit map $u_\alpha$, which is a critical point of the
Sacks-Uhlenbeck functional for $\alpha-1$ sufficiently small.
\end{thm}

  For each $\alpha >1$, let $u_\alpha(x,t)$  be the global
smooth solution to (1.1)   in the same homotopy class with the
 initial map $u_0$ from $M$ to $N$. There is a nature problem
to study the limit behaviour of the solution $u_\alpha (x,t)$ as
$\alpha\to 1$ as one in \cite {SU}.
 Thus,
we prove
\begin{thm}
    \label{thm:routine}

    (i) As $\alpha \to 1$, the solution $u_\alpha$ converges smoothly to $u$ on $M\times [0,\infty)\setminus \Sigma$, where
     the concentration set $\Sigma$  is a closed set defined by
\begin{equation*}
    \Sigma=\bigcap_{0<R<R_M} \left\{z\in M\times [0,\infty)|\quad \liminf_{\alpha\to 1} \Psi^\alpha_R(u_\alpha,z) \geq \varepsilon_0 \right\}
\end{equation*}
for some $\varepsilon_0$ to be determined. For the precise
definition of $\Psi_R^\alpha$, see Section \ref{sec:approx}.

(ii) For any two positive $t_1$ and $t_2$,   $\mathcal P^2 (\Sigma\cap (M\times
[t_1,t_2]))$ is finite, where $\mathcal P^2$ denotes the
$2$-dimensional parabolic Hausdorff measure. Moreover, for any
$t\in (0,+\infty)$, $\Sigma_t=\Sigma\cap (M\times \{t\})$
     consists of  at most finitely many points.

     (iii)  $u$ is a weak solution to the harmonic map flow.
\end{thm}

In fact,  this result is similar to the one by the Ginzburg-Landau
flow  approximation.
 In \cite{CS}, Chen and Struwe used the Ginzburg-Landau flow  approximation  to construct a global weak solution of the harmonic map
 flow for any dimension larger than two.
Further results obtained by  the Ginzburg-Landau flow were
discussed by Lin and Wang in \cite{LW}. The method by the
Ginzburg-Landau flow is very powerful to show the global existence
and partial regularity of a weak solution to the harmonic map
flow, but it seems that the flow loses control on the topological
quantity of maps. The advantage of the Sacks-Uhlenbeck flow
approximation is
 that the solution $u_\alpha$ remains in the same homotopy class of $u_0$, therefore it seems that the Sacks-Uhlenbeck
flow provides a nice geometric picture and can be used to have
some  geometric applications (see Theorem \ref{thm:lin} below).

Without the assumption of the energy inequality,
 weak solutions to the heat flow for harmonic maps may  not always
 be
unique (see \cite{unique1} and \cite{unique2}), so it is very
interesting whether the limiting solution $u$ by the
Sacks-Uhlenbeck flow is the global weak solution $\hat{u}$
constructed by   Struwe in \cite{struwe}.  Although we do not have
a complete answer to the question,  we can compare some property
of $u$ with the one of $\hat{u}$ in the following:

\begin{thm}\label{thm:compare}
(i) The first concentration time for $u_\alpha$ as $\alpha\to 1$;
i.e.
    \begin{equation*}
        T=\inf_{(x,t)\in \Sigma} t
    \end{equation*}
    is the same as the first singular time of the Struwe's weak solution. Hence, $u_\alpha$ converges smoothly to the Struwe's solution in $M\times [0, T)$.

    (ii) Consider the following two limits of measures
    \begin{equation*}
       |\nabla \hat{u}(\cdot,t)|^2 dv \to |\nabla \hat{u}(\cdot, T)|^2 dv+\sum_{i=1}^{\hat{k}} \hat{m}_i \delta_{\hat{p}_i} \quad \mbox{as } t\to T
    \end{equation*}
    and
    \begin{equation*}
      (1+|\nabla u_\alpha(T,\alpha)|^2)^{\alpha} dv \to (1+|\nabla u(\cdot,T)|^2)\,dv +\sum_{i=1}^k m_i \delta_{p_i} \quad \mbox{as } \alpha\to 1.
    \end{equation*}
    Then $k=\hat{k}$, $p_i=\hat{p}_i$ and $\hat{m}_i\geq m_i$.
\end{thm}
Moreover, we can obtain more refined information about the
singularity points of $u$ and $\Sigma$ (see Lemma
\ref{thm:right}).

In the final part of this paper, we apply the $\alpha-$flow to
study a minimizing sequence in a given homotopy class. Let $u_i$
be a sequence of smooth maps minimizing $E(u)=\int_M \abs{\nabla
u}^2 dv$ in  a fixed homotopy class of maps. Since $u_i$ is
bounded in $W^{1,2}$, there is a weak limit $u$ in $W^{1,2}(M,N)$.
In general, $u$ may not be in the same homotopy class, but we can
show:
\begin{thm}\label{thm:lin}
Let $u$ be the weak limit of  above minimizing sequence $\{ u_i
\}$. Then it is a harmonic map from $M$ to $N$ and there exist
harmonic maps $\omega_k: S^2\to N$ with $k=1,\cdots,l$ such that
  \begin{equation}\label{eqn:identity}
      \lim_{i\to \infty} E(u_i)=E(u) +\sum_{k=1}^l E(\omega_k).
  \end{equation}
  Moreover, if $\pi_2(N)$ is trivial, then $u_i$ converges strongly to $u$ in $W^{1,2}(M,N)$ and $u$ is a minimizer in the homotopy class of $u_i$.
\end{thm}

The last part of this theorem can be compared with Theorem 5.1 of
\cite{SU}, where under the same assumption that $\pi_2(N)$ is
trivial, the existence of a minimizer in each homotopy class  is
proved (see also \cite{SY} and \cite{Lemaire}). We improve the
result a little by obtaining that every smooth minimizing sequence
converges to such a minimizer. In the proof, we use the
$\alpha-$flow to modify the original minimizing sequence and study
the blow-up of the new sequence as Theorem 5.1 in \cite{SU}. The
energy identity of minimizers of the Sacks-Uhlenbeck functional was
implicitly established by Chen and Tian in \cite {CT}. Our proof
for the energy identity is different from one in \cite{CT}. Although
the result of Theorem \ref{thm:lin} may be regarded
as a consequence of the theory developed by Duzaar and Kuwert
\cite{DK}, the advantage is that we can avoid to use the
concept ``weak homotopy class''  of maps in the Sobolev space
$W^{1,2}(M, N)$, which is formulated by big machinery. Finally, we
would like to mention that  the related energy identity of critical
points of the Sacks-Uhlenbeck functionals was recently discussed
by Li and Wang \cite{LiW} and by Lamm \cite{Lamm}.

\begin{rem}
    After we finished a first version of this paper, Yuxiang Li
    informed
    us that in \cite{Yuxiang} they used a similar idea   to show the energy identify for a sequence of minimizers of $E_\alpha$ when $\alpha\to 1$.
   In fact,  while  Li and Wang in \cite{Yuxiang} used a reduction procedure of Ding and Tian
   \cite{DT},
     we use the bubble tree construction of Parker in \cite{P} and Lemma \ref{lem:smallball} to find the connecting
     geodesics, so our approach is different.
\end{rem}

The paper is organized as follows. In Section \ref{sec:global},  we prove Theorem
1.1. In Section \ref{sec:approx}, we give a proof of Theorem 1.2. In Section \ref{sec:lin},
we complete a proof of Theorem 1.3 except for the energy identity, which is proved
in Section \ref{sec:identity}.

\section{global existence and convergence of $\alpha-$flow} \label{sec:global}
This section is devoted to the proof of Theorem \ref{thm:global}.
In local coordinates $g=(g_{ij})$, the $\alpha$-flow can be
written as
\begin{equation*}
        \partial_t u =\frac{1}{\sqrt{\abs{g}}} \pfrac{}{x_i} \left( \sqrt{\abs{g}} g^{ij} \frac {\partial u}{\partial x_j} \right)+
        (\alpha-1) \frac {g^{ij} \pfrac {}{x_i} \abs{\nabla u}^2  \frac {\partial u}{\partial x_j}}{1+\abs{\nabla u}^2} +A(u)(\nabla u,\nabla u).
    \end{equation*}
This system is a nonlinear parabolic system.  For a smooth initial
value $u_0$,  the local existence of the system  can be shown (see
below appendix for details in Section 6); i.e., there exists $T>0$
and a smooth solution $u(x,t)$ defined on $[0,T)$. The proof of
Theorem \ref{thm:global} follows if one can establish $C^k$
uniform estimates of $u(x,t)$ independent of $T$ for any $t\in
[0,T)$.

The first observation is that
\begin{lem}\label{lem:decrease}
    If $u(x,t)$ is a solution to the $\alpha-$flow, then $E_\alpha(u(\cdot,t))$ decreases in $t$.
\end{lem}
\begin{proof}
    Multiply (\ref{eqn:flow1}) by $(1+\abs{\nabla u}^2)^{\alpha-1} \partial_t u$ and integrate by parts.
\end{proof}

The next lemma is a local version of the energy inequality.
\begin{lem}
    Let $u$ be a solution to the $\alpha-$flow.  Then
\begin{equation}\label{eqn:local}
    \int_{B_R(x)} (1+\abs{\nabla u(\cdot, t_2)}^2)^\alpha dv \leq \int_{B_{2R}(x)} (1+\abs{\nabla u(\cdot,t_1)}^2)^\alpha dv +C\frac{t_2-t_1}{R^2}E_0
\end{equation}
for $t_1<t_2$. Here $E_0$ is some upper-bound of the overall
energy.
\end{lem}
\begin{proof} Let $\varphi$ be a cut-off function supported in $B_R(x)$ and $\varphi\equiv 1$ on $B_{R/2}(x)$.
Multiplying the equation (\ref{eqn:flow1}) by $(1+\abs{\nabla u}^2)^{\alpha-1}\partial_t u \varphi^2$, we obtain
\begin{eqnarray*}
    &&\int_M (1+\abs{\nabla u}^2)^{\alpha-1} \abs{\partial_t u}^2 \varphi^2 dv\\
    &=& -\int_M (1+\abs{\nabla u}^2)^{\alpha-1} \nabla u\nabla \pfrac{u}{t}\varphi^2
    dv
     -\int_M (1+\abs{\nabla u}^2)^{\alpha-1} \nabla u \pfrac{u}{t} 2 \varphi \nabla \varphi dv\\
    &\leq& -\frac{1}{\alpha} \frac{d}{dt} \int_M (1+\abs{\nabla u}^2)^\alpha \varphi^2
    dv
     +\frac{1}{2}\int_M (1+\abs{\nabla u}^2)^{\alpha-1} \abs{\partial_t u}^2 \varphi^2 dv \\
    && + \int_{B_{2R(x)}} (1+\abs{\nabla u}^2)^{\alpha-1} \abs{\nabla u}^2 \frac{C}{R^2} dv.
\end{eqnarray*}
Hence,
\begin{equation*}
    \int_{B_R(x)} (1+\abs{\nabla u(\cdot, t_2)}^2)^\alpha dv \leq \int_{B_{2R}(x)} (1+\abs{\nabla u(\cdot,t_1)}^2)^\alpha dv +C\frac{t_2-t_1}{R^2}E_0.
\end{equation*}
\end{proof}
The second  key  for the proof of Theorem \ref{thm:global} is  to
derive a Bochner type formula. We consider a scaled version of
(\ref{eqn:flow1}) for some $r>0$,
\begin{equation}\label{eqn:flow2}
    (r^2+\abs{\nabla u}^2)^{\alpha-1}\partial_t u= \mbox{div} \left( (r^2+\abs{\nabla u}^2)^{\alpha-1} \nabla u \right)+(r^2+\abs{\nabla u}^2)^{\alpha-1}A(u)(\nabla u,\nabla u).
\end{equation}
Locally, we choose an orthonormal frame $\{e_1,e_2\}$. We use
$\nabla_i$ for the first covariant derivative with respect to
$e_i$. We denote   by $u_{ji}$  the second covariant derivatives
of $u$ and so on.  Of course, we assume summation convention for
repeated index. Then
\begin{lem} \label{lem:bochner}
    Let $u(x,t)$ be a smooth solution to the scaled $\alpha-$flow (\ref{eqn:flow2}). If $\alpha-1$ is small, then the following Bochner type formula is true:
    \begin{equation}\label{eqn:bochner}
        \frac{\partial}{\partial t}e(u)-\nabla_i\left( (\delta_{ij}+2(\alpha-1)\frac{u^\beta_i u^\beta_j}{r^2+\abs{\nabla u}^2}) \nabla_j e(u) \right)\leq C e(u)(e(u)+1),
    \end{equation}
    where $e(u):=\abs{\nabla u}^2$.
\end{lem}

\begin{proof}
  The proof is by computation.
In the following proofs, we assume $\alpha-1$ is small whenever
necessary. In a local frame, we have
\begin{equation*}
       \nabla_j e(u)=2u^{\gamma}_{k}u^{\gamma}_{kj},\quad
       |\nabla^2 u|^2=\sum_{k,i, \gamma} \abs{u^\gamma_{ki}}^2
    \end{equation*}
Then we have
    \begin{eqnarray*}
        &&\nabla_i\left( (\delta_{ij}+2 (\alpha-1)\frac{u^\beta_i u^\beta_j}{r^2+\abs{\nabla u}^2}) \nabla_j e(u) \right)\\
        &=& 2\nabla_i \left( u^\gamma_k u^\gamma_{ki} +2 (\alpha-1)\frac{ u^\beta_i u^\beta_j u^\gamma_k u^\gamma_{kj}}{r^2+\abs{\nabla u}^2} \right) \\
        &\geq& 2 |\nabla^2 u|^2+2 u^\gamma_k u^\gamma_{iik}+  4(\alpha-1)\nabla_i \left( \frac{ u^\gamma_i u^\gamma_k u^\beta_j u^\beta_{kj}}{r^2+\abs{\nabla u}^2} \right) -Ce(u)\\
        &\geq& \frac{3}{2} |\nabla^2 u|^2 +2u^\gamma_k u^\gamma_{iik}+  4(\alpha-1) \frac{ u^\gamma_i u^\gamma_k u^\beta_j u^\beta_{kji}}{r^2+\abs{\nabla u}^2}-Ce(u) \\
        &\geq&  |\nabla^2 u|^2+2u^\gamma_k \nabla_k \left( u^\gamma_{ii}+ 2 (\alpha-1) \frac{u^\gamma_i u^\beta_j u^\beta_{ji}}{r^2+\abs{\nabla u}^2}\right)-Ce(u).
    \end{eqnarray*}
    Here we have used twice Ricci identity for switching third order derivatives. Using (2.2), we have
    \begin{eqnarray*}
        &&\frac{\partial}{\partial t} e(u)-\nabla_i \left( (\delta_{ij}+ 2(\alpha-1)\frac{u^\beta_i u^\beta_j}{r^2+\abs{\nabla u}^2}) \nabla_j e(u) \right) \\
        &\leq& - |\nabla^2 u|^2  -2u_k^\gamma \nabla_k \left( A^\gamma(u)(\nabla u,\nabla u) \right)+Ce(u) \\
        &\leq& C e(u)(e(u)+1).
    \end{eqnarray*}
\end{proof}

Using this Bochner formula, we can prove a small energy
estimate following a method of Schoen \cite{Sch} and Struwe
\cite{mono}. In our case, the small energy assumption is
automatically true because of the H\"older inequality and the
$E_\alpha$ energy bound. More precisely, we have
\begin{lem} \label{lem:Linfinity}
    There is a constant $C$ independent of $(x,t)$ for $x\in M$ and $t\in [0,T)$ such that
    \begin{equation*}
        \abs{\nabla u}(x,t)\leq C.
    \end{equation*}
\end{lem}

\begin{proof}
    For any $t\in [0,T)$, Lemma \ref{lem:decrease} implies that
    \begin{equation*}
        \int_M (1+\abs{\nabla u}^2)^\alpha(t) dv\leq C.
    \end{equation*}
   By the H\"older inequality, there exists a uniform constant $r_0>0$ such that for all $x\in M$ and $t<T$,
    \begin{equation*}
        \int_{B_{r_0}(x)} \abs{\nabla u(\cdot,t)}^2 dv\leq \varepsilon_0.
    \end{equation*}
    Here $\varepsilon_0$ is a small constant to be determined. Since $u$ is smooth in $M\times [0,T)$, it suffices to prove the lemma near $t=T$. By choosing $r_0$ sufficiently small, we may assume $r_0^2<T$. Take any $x_0\in M$ and $t_0\in(r^2_0,T)$. Set
    \begin{equation*}
        P_r=\{(x,t)|x\in B_r(x_0) \quad \mbox{and } t_0-r_0^2\leq t\leq t_0\}.
    \end{equation*}
    We can find $\rho\in [0,r_0]$ such that
    \begin{equation*}
        (r_0-\rho)^2 \sup_{P_\rho} e(u)=\max_{\sigma\in [0,r_0]} \left\{ (r_0-\sigma)^2 \sup_{P_\sigma}e(u)\right\}.
    \end{equation*}
    Let $(x_1,t_1)$ be the point in $P_\rho$ such that
    \begin{equation*}
        e\stackrel{\triangle}{=}e(u)(x_1,t_1)=\sup_{P_\rho} e(u).
    \end{equation*}
    If $e(r_0-\rho)^2\leq 4$, then
    \begin{equation*}
        (\frac{r_0}{2})^2 \sup_{P_{r_0/2}} e(u) \leq e(r_0-\rho)^2 \leq 4,
    \end{equation*}
    which means the lemma is proved. Hence, we may assume that $e(r_0-\rho)^2>4$. Set
    \begin{equation*}
        v(x,t)=u(x_1+\frac{x}{e^{1/2}},t_1+\frac{t}{e})
    \end{equation*}
    for $(x,t)\in B_1(0)\times [-1,0]$. By our definition of $v$ and the scaling invariance of Dirichlet energy, we have
    \begin{equation}\label{eqn:moser}
        \sup_{t\in [-1,0]}\int_{B_1(0)} e(v) dv \leq \varepsilon_0.
    \end{equation}
    Moreover, we have $e(v)(0,0)=1$ and
    \begin{equation}\label{eqn:four}
        \sup_{B_1(0)\times [-1,0]} e(v)\leq e^{-1}\sup_{P_{\frac{r_0+\rho}{2}}} e(u) \leq e^{-1} \frac{e (r_0-\rho)^2}{(r_0-\frac{r_0+\rho}{2})^2}\leq 4.
    \end{equation}
    By Lemma \ref{lem:bochner} and (\ref{eqn:four}),
    \begin{equation*}
        \pfrac{}{t}e(v)-\nabla_i (a_{ij}(v) \nabla_j e(v)) \leq Ce(v),
    \end{equation*}
    where
\[a_{ij}(v)=\delta_{ij}+2(\alpha-1)\frac{v^\beta_i v^\beta_j}{r^2+\abs{\nabla
v}^2}.
\]
The symmetric matrix $(a_{ij}(v))$  has eigenvalues satisfy the
uniform elliptic condition. In local coordinates, we can write the
above inequality as
    \begin{equation*}
        \partial_t e(v) -\frac{1}{\sqrt{\abs{g}}} \pfrac{}{x_i} \left( \sqrt{\abs{g}} (g^{ij}+\frac{(\alpha-1)}{r^2+\abs{\nabla v}^2}
        g^{ik}\pfrac{v^\beta}{x_k} g^{jl}\pfrac{v^\beta}{x_l} ) \pfrac{e(v)}{x_j}  \right) \leq Ce(v).
    \end{equation*}
    By a standard Moser iteration and (\ref{eqn:moser}), we have
    \begin{equation*}
        1=e(v)(0,0)\leq C \left( \int_{B_1(0)\times [-1,0]} e(v)^2 \right)^{1/2}.
    \end{equation*}
    By (\ref{eqn:four}), we obtain
    \begin{equation*}
        1\leq C\int_{B_1(0)\times [-1,0]} e(v) \leq C\varepsilon_0,
    \end{equation*}
    which is a contradiction if we choose $\varepsilon_0$ small.
\end{proof}

We now complete the proof of Theorem \ref{thm:global}.
\begin{proof}
    Lemma \ref{lem:Linfinity} provides the $C^0-$uniform estimate for $u(x,t)$. Then we show higher order estimates. For $x_0\in M$ and $T/2<t_0<T$, set $P_r=B_r(x_0)\times [t_0-r^2,t_0]$. Suppose $\varphi$ is a cut-off function supported in $P_{T/2}$ and $\varphi\equiv 1$ in $P_{T/4}$. Multiplying (\ref{eqn:flow1}) by $\varphi$, we obtain
    \begin{equation}\label{eqn:cutoff}
        \partial_t(u\varphi)-\triangle(u\varphi)=(\alpha-1)\frac{(\nabla ^2(u\varphi),\nabla u)\nabla u}{1+\abs{\nabla u}^2} +\mathcal{R}[\varphi,u,\nabla u],
    \end{equation}
    where $\mathcal R$ is a term involving $u,\nabla u$ and derivatives of $\varphi$.
    By Lemma \ref{lem:Linfinity} and the $L^p$ estimate of linear parabolic equations, we have
    \begin{equation*}
        \norm{\varphi u}_{W^{2,p}(P_{T/2})}\leq C(\alpha-1)\norm{\nabla ^2(\varphi u)}_{L^p(P_{T/2})} + C,
    \end{equation*}
    for any $p>0$.
    When $\alpha-1$ is small,
    \begin{equation*}
        \norm{\varphi u}_{W^{2,p}(P_{T/2})}\leq  C.
    \end{equation*}
    We then take one more space derivative of (\ref{eqn:cutoff}) to get
    \begin{equation*}
        \partial_t(\varphi \nabla u)-\triangle (\varphi \nabla u)=(\alpha-1)\frac{(\nabla ^3(\varphi u),\nabla u)\nabla u}{1+\abs{\nabla u}^2}+\mathcal Q[\varphi,u,\nabla u,\nabla ^2u].
    \end{equation*}
    Here $\mathcal Q$ involves $u,\nabla u,\nabla ^2u$ and derivatives of $\varphi$. Using $L^p$ estimate again, we have (when $\alpha-1$ is sufficeintly small)
    \begin{equation*}
        \norm{\varphi u}_{W^{3,p}(P_{T/2})}\leq C,
    \end{equation*}
    for any $p$. Therefore,
    \begin{equation*}
        \norm{\partial_t \nabla (\varphi u)}_{L^{p}(P_{T/2})},\quad \norm{ \nabla ^2(\varphi u)}_{L^{p}(P_{T/2})}\leq C.
    \end{equation*}
    For $p>3$, the Sobolev embedding theorem yields that
    \begin{equation*}
        \norm{\nabla (\varphi u)}_{C^{\beta}(P_{T/2})}\leq C
    \end{equation*}
    for some $\beta>0$. We can now apply the Schauder theory for parabolic equations to obatin higher order estimates, which means that we can extend the solution $u(x,t)$ smoothly to $t=T$. Then the local existence result for smooth initial data implies $T=\infty$.
\end{proof}

\section{The limiting behaviour of the Sacks-Uhlenbeck flow}\label{sec:approx}
For a fixed initial map $u_0$, it follows from  Theorem
\ref{thm:global}  that there is a unique global solution
$u_\alpha$ to (\ref{eqn:flow1})in $M\times [0,\infty )$ for each
$\alpha>1$. In this section, we study the limit of  the solutions
$u_\alpha$ as $\alpha\to 1$. We start with two lemmas. The first
one is another Bochner type formula.

\begin{lem}
  \label{lem:bochner1}
    Let $u(x,t)$ be a classical solution to the scaled $\alpha-$flow (\ref{eqn:flow2}). If $\alpha$ is small, then
    there is a constant $C>0$ such that
    \begin{equation}\label{eqn:bochner1}
      \frac{\partial}{\partial t}e_\alpha(u)-\nabla_i\left( (\delta_{ij}+2(\alpha-1)\frac{u^\beta_i u^\beta_j}{r^2+\abs{\nabla u}^2}) \nabla_j e_\alpha(u) \right)\leq C e_\alpha(u)(1+\abs{\nabla u}^2),
    \end{equation}
    where $e_\alpha(u):=(r^2+\abs{\nabla u}^2)^\alpha$.
\end{lem}
\begin{proof}
  Again, we assume $\alpha-1$ is small whenever necessary.
    \begin{eqnarray*}
        &&\nabla_i\left( (\delta_{ij}+2(\alpha-1)\frac{u^\beta_i u^\beta_j}{r^2+\abs{\nabla u}^2}) \nabla_j e_\alpha(u) \right)\\
        &=& \nabla_i \left( \alpha(r^2+\abs{\nabla u}^2)^{\alpha-1} (\delta_{ij}+2(\alpha-1)\frac{u^\beta_i u^\beta_j}{r^2+\abs{\nabla u}^2}) 2u^\gamma_k u^\gamma_{kj} \right) \\
        &=& 2\alpha \nabla_i \left( (r^2+\abs{\nabla u}^2)^{\alpha-1}u^\gamma_k u^\gamma_{ki}+2(\alpha-1)(r^2+\abs{\nabla u}^2)^{\alpha-2} u^\beta_i u^\beta_j u^\gamma_k u^\gamma_{kj} \right) \\
        &=& 2\alpha \nabla_i \left( (r^2+\abs{\nabla u}^2)^{\alpha-1}u^\gamma_k u^\gamma_{ki}+2(\alpha-1)(r^2+\abs{\nabla u}^2)^{\alpha-2} u^\gamma_i u^\gamma_k u^\beta_j u^\beta_{kj} \right) \\
        &=& 2\alpha \nabla_i \left( (r^2+\abs{\nabla u}^2)^{\alpha-1}u^\gamma_k u^\gamma_{ki}+ u^\gamma_i u^\gamma_k \nabla_k (r^2+\abs{\nabla u}^2)^{\alpha-1} \right) \\
        &=& 2\alpha \nabla_i \left[ u^\gamma_k \nabla_k (u^\gamma_i (r^2+\abs{\nabla u}^2)^{\alpha-1}) \right]\\
        &=& 2\alpha \abs{u^\gamma_{ki}}^2 (r^2+\abs{\nabla u}^2)^{\alpha-1} +4\alpha(\alpha-1)(r^2+\abs{\nabla u}^2)^{\alpha-2}u^\gamma_{ki}u^{\beta}_{jk}u^\gamma_i u^\beta_j \\
        && + 2\alpha u^\gamma_k \nabla_i\nabla_k (u^\gamma_i (r^2+\abs{\nabla u}^2)^{\alpha-1})\\
        &\geq & \alpha \abs{u^\gamma_{ki}}^2(r^2+\abs{\nabla u}^2)^{\alpha-1} +2\alpha u_k^\gamma \nabla_k \nabla_i (u_i^\gamma (r^2+\abs{\nabla u}^2)^{\alpha-1}) -C (r^2+\abs{\nabla u}^2)^{\alpha} \\
        &\geq& \alpha \abs{u^\gamma_{ki}}^2(r^2+\abs{\nabla u}^2)^{\alpha-1} +2\alpha (r^2+\abs{\nabla u}^2)^{\alpha-1} u^\gamma_k\nabla_k \pfrac{u^\gamma}{t}\\
        && 2\alpha(r^2+\abs{\nabla u}^2)^{\alpha-1}u^\gamma_k \nabla_k (A^\gamma(u)(\nabla u,\nabla u)) -C(r^2+\abs{\nabla u}^2)^\alpha \\
        &\geq& \frac{1}{2}\abs{u^\gamma_{ki}}^2 (r^2+\abs{\nabla u}^2)^{\alpha-1}+ \frac{\partial }{\partial t}e_\alpha(u) -Ce_\alpha(u)(1+\abs{\nabla u}^2).
    \end{eqnarray*}
\end{proof}

The second lemma is a parabolic monotonicity formula. Such a
formula for the harmonic map flow was first established by Struwe
in \cite{mono} for the Euclidean case and by Chen and Struwe in
\cite{CS} for general case. Let $u$ be a solution to
(\ref{eqn:flow1}). Suppose $R_M$ is the injectivity radius of $M$.
For a fixed point $p$ in $M$, choose the normal coordinates $\{x_i\}$ and a
cut-off function $\varphi$ supported in $B_{R_M}(p)$ such that
$0\leq \varphi\leq 1$ and $\varphi\equiv 1$ in a neighborhood of
$p$. For $z=(p,t_0)$ with some $t_0\in [0,\infty)$, set
\begin{equation*}
    G_z(x,t)=\frac{1}{\abs{t-t_0}}e^{\frac{-\abs{x}^2}{4\abs{t-t_0}}}
\end{equation*}
and
\begin{equation*}
    \Psi^\alpha_\rho(u,z)=\rho^{2\alpha-2}\int_{T_\rho(t_0)} (1+\abs{\nabla u}^2)^{\alpha} G_z \varphi^2 \sqrt{\abs{g}} dx
    dt,
\end{equation*}
where $\sqrt{\abs{g}}dx$ is the volume element of $M$ and
$T_\rho(t_0)=M\times [t_0-4\rho^2,t_0-\rho^2]$.
\begin{lem}\label{lem:mono}
    Let $u$ be a smooth solution to (\ref{eqn:flow1}) defined on $M\times [0,\infty)$ and $E_\alpha(u(t))\leq E_0$. Then for $z$ as above and for any $0<r<\rho\leq R_M$,
    \begin{equation*}
        \Psi^\alpha_r(u,z)\leq e^{c(\rho-r)}\Psi^\alpha_\rho(u,z)+cE_0 (\rho-r),
    \end{equation*}
    with a uniform constant $c$ depending only on $M$ and $N$.
\end{lem}
\begin{proof} After a translation, we can assume $t_0=0$ and write $\Psi_\rho(u)$ for $\Psi^\alpha_\rho(u,z)$.
    Letting $u_\rho(y,s)=u(\rho y, \rho^2 s)$, we obtain
    \begin{equation*}
        \Psi_\rho(u)=\int_{T_1} (\rho^2+g^{ij}(\rho y)\partial_{y_i}u_\rho\partial_{y_j}u_\rho)^\alpha G \varphi^2(\rho y)\sqrt{\abs{g}}(\rho y) dy ds.
    \end{equation*}
    In the following, for simplicity we write $g^{ij}$ for $g^{ij}(\rho y)$ and the same convention applies to $\varphi$, $\bar{\nabla} g^{ij}$ and so on. Here we use $\bar{\nabla}$ to denote the gradient in local coordinates; i.e. $(\pfrac{}{x_1},\pfrac{}{x_2})$ and $\tilde{\nabla}$ to denote $(\pfrac{}{y_1},\pfrac{}{y_2})$. Denoting $\frac{\partial }{\partial y_i} u_\rho$ by $u_{\rho,i}$, we have
    \begin{eqnarray*}
        &&\frac{d}{d\rho}\Psi_\rho(u)\\
        &=& \int_{T_1} \alpha (\rho^2+g^{ij}u_{\rho,i} u_{\rho,j})^{\alpha-1} (2\rho+y \cdot \bar{\nabla} g^{ij} u_{\rho,i} u_{\rho,j}+2 g^{ij} u_{\rho,i} \frac{du_{\rho,j}}{d\rho})
         \cdot G \varphi^2 \sqrt{\abs{g}} dy\,ds \\
         && +\int_{T_1} (\rho^2+g^{ij}u_{\rho,i} u_{\rho,j})^\alpha G \frac{d}{d\rho}(\varphi^2 \sqrt{\abs{g}}) dy\,ds \\
         &\geq &-\int_{T_1}2\alpha \frac{\partial}{\partial y_j} \left( \sqrt{\abs{g}} (\rho^2+g^{ij}u_{\rho,i}u_{\rho,j})^{\alpha-1} g^{ij} u_{\rho,i} G \right) \frac{du_\rho}{d\rho} \varphi^2dy\,ds \\
     &&-\int_{T_1} 4\alpha (\rho^2+g^{ij}u_{\rho,i}u_{\rho,j})^{\alpha-1} g^{ij} u_{\rho,i} \partial_{y_j}\varphi G \frac{du_\rho}{d\rho} \varphi \sqrt{\abs{g}} dy\,ds \\
        &&+ \int_{T_1} \alpha (\rho^2+g^{ij}u_{\rho,i}u_{\rho,j})^{\alpha-1} y\cdot \bar{\nabla} g^{ij} u_{\rho,i} u_{\rho,j} G \varphi^2\sqrt{\abs{g}} dy\,ds \\
        &&+ \int_{T_1} (\rho^2+g^{ij}u_{\rho,i}u_{\rho,j})^\alpha G 2\varphi y\cdot \bar{\nabla} \varphi \sqrt{\abs{g}} dy\,ds \\
        &&+ \int_{T_1} (\rho^2+g^{ij}u_{\rho,i}u_{\rho,j})^\alpha G \varphi^2 \frac{y\cdot \bar{\nabla} \abs{g}}{2\abs{g}} \sqrt{\abs{g}} dy\,ds \\
        &=& I+II+III+IV+V.
    \end{eqnarray*}
By definition of $u_\rho$,
\begin{equation*}
    \frac{du_\rho}{d\rho}=y\cdot \bar{\nabla} u +2\rho s\frac{\partial u}{\partial t}=\frac{2s}{ \rho}(\frac{1}{2s}y\cdot \tilde{\nabla} u_\rho+\pfrac{u_\rho}{s}).
\end{equation*}
Since $u_\rho$ satisfies,
\begin{eqnarray*}
    (\rho^2+g^{ij}u_{\rho,i}u_{\rho,j})^{\alpha-1}\pfrac{u_\rho}{s}&=&  \pfrac{}{y_j} \left( \sqrt{\abs{g}} (\rho^2+g^{ij}u_{\rho,i}u_{\rho,j})^{\alpha-1} g^{ij}u_{\rho,i} \right)\\
    &&+(\rho^2+g^{ij}u_{\rho,i}u_{\rho,j})^{\alpha-1}A(u_\rho)(\tilde{\nabla} u_\rho,\tilde{\nabla} u_\rho),
\end{eqnarray*}
we have
\begin{eqnarray*}
    &I&\geq -2\alpha \int_{T_1} (\rho^2+g^{ij}u_{\rho,i} u_{\rho,j})^{\alpha-1} (\pfrac{u_\rho}{s}-\frac{g^{ij}u_{\rho,i}y_j}{2|s|}) \frac{du_\rho}{d\rho} G \varphi^2\sqrt{\abs{g}} dy\,ds \\
    &\geq & \int_{T_1} \frac{4\alpha |s|}{\rho} (\rho^2+g^{ij} u_{\rho,i} u_{\rho,j})^{\alpha-1} \abs{\pfrac{u_\rho}{s}-\frac{y\cdot\tilde{\nabla}u_\rho}{2\abs{s}}}^2G \varphi^2 \sqrt{\abs{g}} dy\,ds \\
    &-&\int_{T_1} \frac{4\alpha\abs{s} }{\rho} (\rho^2+g^{ij} u_{\rho,i}u_{\rho,j})^{\alpha-1} \abs{\frac{g^{ij}u_{\rho,i}y_j}{2\abs{s}}-\frac{y\cdot\tilde{\nabla}u_\rho}{2\abs{s}}}\abs{\pfrac{u_\rho}{s}-\frac{y\cdot\tilde{\nabla}u_\rho}{2\abs{s}}} G \varphi^2 \sqrt{\abs{g}} dy\,ds \\
    &\geq & \int_{T_1} \frac{2\alpha |s|}{\rho} (\rho^2+g^{ij} u_{\rho,i} u_{\rho,j})^{\alpha-1} \abs{\pfrac{u_\rho}{s}-\frac{y\cdot\tilde{\nabla}u_\rho}{2\abs{s}}}^2G \varphi^2 \sqrt{\abs{g}} dy\,ds \\
    &&- C\int_{T_1} \frac{1}{\rho} (\rho^2+g^{ij} u_{\rho,i}u_{\rho,j})^{\alpha-1} \abs{\frac{g^{ij}u_{\rho,i}y_j}{2}-\frac{y\cdot\tilde{\nabla}u_\rho}{2}}^2 G \varphi^2 \sqrt{\abs{g}} dy\,ds.
\end{eqnarray*}
Since $\{x_i\}$ are normal coordinates, we have $\abs{g^{ij}(\rho y)-\delta^{ij}}\leq C \rho\abs{y}$. The absolute value of the last term is not bigger than
\begin{eqnarray*}
    &&C\int_{T_1} \rho^{-1}(\rho^2+g^{ij}u_{\rho,i}u_{\rho,j})^{\alpha-1} \abs{g^{ij}(\rho y)-\delta^{ij}}^2 \abs{y}^2 \abs{u_{\rho,i}}^2 G \varphi^2 \sqrt{\abs{g}} dy\,ds \\
    &\leq& C\int_{T_1}\rho^{-1}(\rho^2+g^{ij}u_{\rho,i}u_{\rho,j})^{\alpha-1} \rho^2 \abs{y}^4 \abs{u_{\rho,i}}^2 G \varphi^2 \sqrt{\abs{g}} dy\,ds \\
    &\leq& C\int_{T_1} (\rho^2+g^{ij}u_{\rho,i}u_{\rho,j})^\alpha (G+C)\varphi^2 \sqrt{\abs{g}} dy\,ds \\
    &\leq& C\Psi_\rho(u)+C\max_{-4\leq s\leq -1} \int_{M} \rho^{2\alpha-2} (1+\abs{\nabla u}^2)^\alpha \varphi^2 \sqrt{\abs{g}}dx\\
    &\leq& C\Psi_\rho(u)+CE_0,
\end{eqnarray*}
if $\rho<1$. Here we use the fact that for $s\in [-4,-1]$,
$\abs{y}^4 G\leq G+C$. In summary, we have
\begin{eqnarray*}
    I&\geq & \int_{T_1} \frac{2\alpha |s|}{\rho} (\rho^2+g^{ij} u_{\rho,i} u_{\rho,j})^{\alpha-1} \abs{\pfrac{u_\rho}{s}-\frac{y\cdot\tilde{\nabla}u_\rho}{2\abs{s}}}^2G \varphi^2 \sqrt{\abs{g}} dy\,ds \\
    &&- C\Psi_\rho(u)-CE_0.
\end{eqnarray*}
For the remaining terms, we have for a sufficiently small
$\epsilon$
\begin{eqnarray*}
    \abs{II}&\leq& \epsilon\int_{T_1}(\rho^2+g^{ij}u_{\rho,i}u_{\rho,j})^{\alpha-1}\frac{4s^2}{\rho} \abs{\pfrac{u_\rho}{s}-\frac{y\cdot \tilde{\nabla}u_\rho}{2\abs{s}}}^2 G \varphi^2 \sqrt{\abs{g}} dy\,ds  \\
    && + C(\epsilon)\int_{T_1}\rho^{-1} (\rho^2+g^{ij}u_{\rho,i}u_{\rho,j})^{\alpha-1} \left( g^{ij} u_{\rho,i} \partial_{y_j}\varphi \right)^2  \sqrt{\abs{g}} dy\,ds \\
    &\leq& \epsilon\int_{T_1}(\rho^2+g^{ij}u_{\rho,i}u_{\rho,j})^{\alpha-1}\frac{4s^2}{\rho} \abs{\pfrac{u_\rho}{s}-\frac{y\cdot \tilde{\nabla}u_\rho}{2\abs{s}}}^2 G \varphi^2 \sqrt{\abs{g}} dy\,ds  \\
    && + C(\epsilon)\int_{T_1} \rho \sum_k \abs{\partial_{x_k}\varphi}^2 (\rho^2+g^{ij}u_{\rho,i}u_{\rho,j})^{\alpha} \sqrt{\abs{g}} dy\,ds\\
    &\leq& \frac{1}{2}I +C \Psi_\rho(u) +CE_0 +C\max_{-4\leq s\leq -1} \int_M \rho^{1+2\alpha-2} (1+\abs{\nabla u}^2)^\alpha \sqrt{\abs{g}} dx \\
    &\leq& \frac{1}{2}I +C \Psi_\rho(u) +CE_0.
\end{eqnarray*}
Here we use again $\rho<1$ and $G$ is bounded on $T_1$. There is some constant $C$ depending only on the geometry of $M$ such that $\bar{\nabla} g^{ij}$ and $\bar{\nabla} \abs{g}$ are bounded by $C$. Therefore
we have
\begin{equation*}
    \abs{III} +\abs{V}\leq C\Psi_\rho(u) +CE_0
\end{equation*}
and
\begin{eqnarray*}
    \abs{IV}&\leq& \frac{1}{2}\Psi_\rho(u) + C\int_{T_1} (\rho^2+g^{ij}u_{\rho,i}u_{\rho,j})^{\alpha} \abs{y\cdot \bar{\nabla} \varphi}^2 G \sqrt{\abs{g}} dy\,ds \\
    &\leq& \frac{1}{2} \Psi_\rho(u) +CE_0,
\end{eqnarray*}
where we used the fact that $\abs{y}^2 G$ is bounded on $T_1$. In
conclusion, we show that for $\rho<1$,
\begin{equation*}
    \frac{d}{d\rho}\Psi_\rho(u) \geq C \Psi_\rho(u) +CE_0.
\end{equation*}
The lemma follows from integrating this differential inequality.
\end{proof}

With these preparation, we now prove Theorem \ref{thm:routine}.
\begin{proof}
    (i) Let $z_i$ be a sequence of points in $\Sigma$ which converges to $\tilde{z}=(\tilde{x},\tilde{t})$. For any $0<R<R_M$, by the definition of $\Sigma$, we have
    \begin{equation*}
        \liminf_{\alpha\to 1} R^{2\alpha-2} \int_{T_R(z_i)} \varphi^2 e_\alpha(u_\alpha) G_{z_i} \sqrt{\abs{g}} dxdt \geq \varepsilon_0.
    \end{equation*}
    Since $\int_{M}e_\alpha(u_\alpha) \sqrt{\abs{g}} dx$ is uniformly bounded  by $E_0$ and $G_{z_i}$ converges to $G_{\tilde{z}}$ uniformly
    away from $\tilde{z}$, we can take $i\to \infty$ in the above inequality and switch the order of limits to get
    \begin{equation*}
        \liminf_{\alpha\to 1} R^{2\alpha-2} \int_{T_R(\tilde{z})} \varphi^2 e_\alpha(u_\alpha) G_{\tilde{z}} \sqrt{\abs{g}} dxdt \geq \varepsilon_0.
    \end{equation*}
    Since $R$ is arbitrary, we know $\tilde{z}\in \Sigma$, which shows that $\Sigma$ is closed.

    For any $R>0$, since $M\times [t^1,t^2]$ is compact, we can find a finite cover of $\Sigma$; i.e., $z_i=(x_i,t_i)\in M\times [t^1,t^2]$  such that
    \begin{equation*}
        \Sigma\cap (M\times [t^1,t^2] ) \subset \bigcup_{i=1}^l Q_{5R}(z_i)
    \end{equation*}
    and $\{Q_{R}(z_i)\}$ are disjoint. Here $Q_r(z)$ is defined to be $B_r(x)\times (t-r^2,t+r^2)$.
    Let $\bar{z}_i=(x_i,t_i+R^2)$.
    For some constant $\delta\in (0,1/4)$ to be determined later, we have (by definition of $\Sigma$ and for $\alpha-1$ sufficiently small, in fact, for each $i$, $\alpha-1$ is required to be smaller)
    \begin{eqnarray}\label{eqn:small1}
        \frac{1}{2}\varepsilon_0 &\leq&  (\delta R)^{2\alpha-2} \int_{t_i-4\delta^2R^2}^{t_i- \delta^2 R^2} \int_M \varphi^2 e_\alpha(u_\alpha) G_{z_i} \sqrt{\abs{g}} dxdt \\ \nonumber
        &\leq& C(\delta) R^{2\alpha-4} \int_{Q_{R}(z_i)} e_\alpha(u_\alpha) \sqrt{\abs{g}}dxdt \\ \nonumber
        && +C \delta^{-2}e^{-\frac{1}{16 \delta^2}} (\delta R)^{2\alpha-2} \int_{t_i-4\delta^2 R^2}^{t_i-\delta^2R^2} \int_M \varphi^2 e_\alpha(u_\alpha) G_{\bar{z}_i} \sqrt{\abs{g}}dxdt
    \end{eqnarray}
    for $1\leq i\leq l$. Here in the last step above, we use the fact that
    \begin{equation*}
        G_{z_i}\leq \delta^{-2} e^{-\frac{1}{16 \delta^2}} G_{\bar{z}_i} \mbox{ in } B_{R_M}(x_i)\times [t_i-4\delta^2 R^2,t_i-\delta^2 R^2] \setminus Q_{R}(z_i).
    \end{equation*}
    By Lemma \ref{lem:mono}, the second term in the right hand side of (\ref{eqn:small1}) can be estimated as
    \begin{eqnarray*}
        C\delta^{-2} e^{-\frac{1}{16\delta^2}} \Psi^\alpha_{\delta R}(u_\alpha,\bar{z}_i) &\leq& C\delta^{-2} e^{-\frac{1}{16 \delta^2}} (\Psi^\alpha_{\sqrt{t^1}/4}(u_\alpha,\bar{z}_i) +E_0) \\
        &\leq& C(E_0,t^1) \delta^{-2} e^{-\frac{1}{16 \delta^2}}.
    \end{eqnarray*}
    Hence, we can choose $\delta$ small to make it smaller than $\varepsilon_0/4$. Therefore,
    \begin{equation*}
        R^2\leq R^{4-2\alpha}\leq C\int_{Q_{R}} e_\alpha(u_\alpha) \sqrt{\abs{g}} dxdt.
    \end{equation*}
    Since $Q_{R}(z_i)$ are disjoint, we have
    \begin{equation*}
        \sum_{i=1}^l (5R)^2 \leq C\sum_{i=1}^l \int_{Q_{R}(z_i)} e_\alpha(u_\alpha) \sqrt{\abs{g}} dxdt \leq CE_0.
    \end{equation*}
    By sending $R$ to zero, we see that $\mathcal P^2 (\Sigma\cap (M\times [t^1,t^2]))$ is finite.
    One can prove that $\Sigma_t$ is a finite set in a similar way. If it is an infinite set, then we can find $l$ distinct points $x_1,\cdots,x_l$ in $\Sigma_t$ and $R>0$ such that $B_R(x_i)$ are disjoint. For the same $\delta>0$ as before, we may repeat the above argument for $(x_i,t)$ to see
    \begin{equation*}
      l (5R^2)\leq C\sum_{i=1}^l \int_{Q_R( (x_i,t))} e_\alpha (u_\alpha) \sqrt{\abs{g}} dxdt\leq CE_0 R^2.
    \end{equation*}
    This gives an upper bound on the number of points in $\Sigma_t$.

    (ii) The proof in this part is similar to the proof of Theorem \ref{thm:global}.  When $\alpha$ is fixed, we know that
    there is no concentration for the Dirichlet energy, so we can use the small energy condition with (\ref{eqn:bochner}) to obtain a $C^0$-bound of the gradient
    as in Lemma \ref{lem:Linfinity}.   However,  as $\alpha$ goes to $1$,  the difference is that there may be a concentration point of the Dirichlet
    energy, but
   for $z\notin \Sigma$, we can find $1/2>R>0$ such that
    \begin{equation*}
      \liminf_{\alpha\to 1} \Psi_R^\alpha (u_\alpha,z) <\varepsilon_0.
    \end{equation*}

    {\bf Claim:} There is an $\varepsilon_0>0$ which we use in the definition of $\Sigma$ and some $\delta>0$ and some constant $C$ depending on $N, R, E_0$ and $\varepsilon_0$ such that
    \begin{equation*}
        \norm{\nabla u_{\alpha_k}}_{C^0(P_{\delta R}(z))}\leq C.
    \end{equation*}

    Assume $z=(0,0)$ and set $r_1=\delta R$ for some $\delta<1/2$ to be determined. Suppose $r,\sigma\in (0,r_1)$, $r+\sigma<r_1$ and $z_0=(x_0,t_0)\in P_r$. We need to show
    \begin{equation*}
        \sigma^{2\alpha_k-4}\int_{P_\sigma(z_0)} e_{\alpha_k}(u_{\alpha_k}) dvdt \leq C\varepsilon_0
    \end{equation*}
    if $R$ and $\delta$ small. To see this, Lemma \ref{lem:mono} implies
    \begin{eqnarray}\label{eqn:st1}
        &&\sigma^{2\alpha_k-4}\int_{P_\sigma(z_0)} e_{\alpha_k}(u_{\alpha_k}) dvdt \\ \nonumber
        &\leq& c \sigma^{2\alpha_k-2} \int_{P_\sigma(z_0)} e_{\alpha_k}(u_{\alpha_k}) G_{(x_0,t_0+2\sigma^2)} dv dt \\ \nonumber
        &\leq& c \sigma^{2\alpha_k-2} \int_{T_\sigma(t_0+2\sigma^2)} \varphi^2 e_{\alpha_k}(u_{\alpha_k}) G_{(x_0,t_0+2\sigma^2)} dv dt \\\nonumber
        &\leq&  ce^{cR-\sigma} \Psi^{\alpha_k}_{R}(u_{\alpha_k},(x_0,t_0+2\sigma^2)) +cE_0(R-\sigma) .
    \end{eqnarray}
    We can choose $R$ to be small so that the last term is no larger than $\varepsilon_0$.
    \begin{eqnarray}\label{eqn:st2}
        &&\Psi^{\alpha_k}_R(u_{\alpha_k},(x_0,t_0+2\sigma^2))\\ \nonumber
        &\leq& cR^{2\alpha_{k}-2}\int_{t_0+2\sigma^2-4R^2}^{t_0+2\sigma^2-R^2} \int_M \varphi^2 e_{\alpha_k}(u_{\alpha_k}) G_{(x_0,t_0+2\sigma^2)} dvdt \\\nonumber
        &\leq& c R^{2\alpha_k-2}\int_{-4R^2}^{-R^2}\int_M \varphi^2 e_{\alpha_k}(u_{\alpha_k}) G_{(x_0,t_0+2\sigma^2)} dvdt +(\delta R)^2 E_0 R^{-2} \\\nonumber
        &\leq& c \Psi^{\alpha_k}_R (u_{\alpha_k},(x_0,t_0)) + \varepsilon E_0 +\delta E_0.
    \end{eqnarray}
    Here in the last line we used the fact that on $T_R$, for any $\varepsilon>0$, we can find $\delta$ so small such that
    \begin{equation*}
        G_{(x_0,t_0+2\sigma^2)}\leq cG(x,t)+\varepsilon R^{-2},
    \end{equation*}
    which is (2.18) of \cite{CS}.  Hence, by choosing $\varepsilon,\delta$ properly, we have
    \begin{equation}
        \sigma^{2\alpha_k-4} \int_{P_\sigma (x_0,t_0)} e_{\alpha_k}(u_{\alpha_k}) dvdt \leq C\varepsilon_0
        \label{eqn:formoser}
    \end{equation}
    for any $(x_0,t_0)\in P_r$ and $r+\sigma<r_1=\delta R$.

    We can find $\rho\in [0,r_1]$ such that
    \begin{equation*}
        (r_1-\rho)^{2\alpha_k} \sup_{P_\rho} e_{\alpha_k}(u_{\alpha_k})=\max_{\sigma\in [0,r_1]} \left\{ (r_1-\sigma)^{2\alpha_k} \sup_{P_\sigma}e_{\alpha_k}(u_{\alpha_k})\right\}.
    \end{equation*}
    Let $(x_1,t_1)$ be the point in $P_\rho$ such that
    \begin{equation*}
        e_0\stackrel{\triangle}{=}e_{\alpha_k}(u_{\alpha_k})(x_1,t_1)=\sup_{P_\rho} e_{\alpha_k}(u_{\alpha_k}).
    \end{equation*}
    If $e_0(r_1-\rho)^{2\alpha_k}\leq 4$, then
    \begin{equation*}
        (\frac{r_1}{2})^{2\alpha_k} \sup_{P_{r_1/2}} e_{\alpha_k}(u_{\alpha_k}) \leq e_0(r_1-\rho)^{2\alpha_k} \leq 4,
    \end{equation*}
    which means the claim is true. Hence, we may assume that $e_0(r_1-\rho)^{2\alpha_k}>4$. Set $\lambda=e_0^{\frac{1}{2\alpha_k}}$ and
    \begin{equation*}
        v(x,t)=u_{\alpha_k}(x_1+\frac{x}{\lambda},t_1+\frac{t}{\lambda^2})
    \end{equation*}
    for $(x,t)\in B_1(0)\times [-1,0]$.
    Then $v$ satisfies a scaled $\alpha-$flow equation
    \begin{eqnarray*}
        (\lambda^{-2}+\abs{\nabla v}^2)^{\alpha_k-1}\partial_t v&=& \mbox{div} \left( (\lambda^{-2}+\abs{\nabla v}^2)^{\alpha_k-1} \nabla v \right)\\
        &&+(\lambda^{-2}+\abs{\nabla v}^2)^{\alpha_k-1}A(v)(\nabla v,\nabla v).
      \end{eqnarray*}
    If we write $e(v)$ for $(\lambda^{-2}+\abs{\nabla v}^2)^{\alpha_k}$, then
    we have $e(v)(0,0)=1$ and
    \begin{eqnarray*}\label{eqn:four1}
        \sup_{B_1(0)\times [-1,0]} e(v) &\leq& e_0^{-1}\sup_{P_{\frac{r_1+(2^{1/\alpha_k}-1)\rho}{2^{1/\alpha_k}}}} e_{\alpha}(u_\alpha) \\ \nonumber
        &\leq& e_0^{-1} \frac{e_0 (r_1-\rho)^{2\alpha_k}}{(r_1-\frac{r_1+(2^{1/\alpha_k}-1)\rho}{2^{1/\alpha_k}})^{2\alpha_k}} \\ \nonumber
        &=& \left( \frac{2^{1/\alpha_k}}{2^{1/\alpha_k}-1} \right)^{2\alpha_k} \leq 5
    \end{eqnarray*}
    if $k$ is large.
    By Lemma \ref{lem:bochner1} and (\ref{eqn:four1}),
    \begin{equation*}
        \pfrac{}{t}e(v)-\nabla_i (a_{ij}(v) \nabla_j e(v)) \leq Ce(v),
    \end{equation*}
    where $(a_{ij})$ is a symmetric matrix whose eigenvalues satisfy uniform elliptic condition. By a standard Moser iteration again, we have
    \begin{eqnarray*}
        1&=&e(v)(0,0)\leq C\int_{B_1(0)\times [-1,0]} e(v) \\
        &\leq& C \lambda^{4-2\alpha_k} \int_{t-\lambda^{-2}}^t \int_{B_{1/\lambda}(x)} (1+\abs{du_{\alpha_k}})^{\alpha_k} dv dt \\
        &\leq& C\varepsilon_0.
    \end{eqnarray*}
    Here in the last step, we used (\ref{eqn:formoser}).
    This is a contradiction if we choose $\varepsilon_0$ small. This concludes the proof of the claim.

    We can establish higher order estimates for $u_{\alpha_k}$ in a smaller neighborhood as in the proof of Theorem \ref{thm:global} so that (ii) follows.

    (iii) Applying the uniform bound of $\nabla u_{\alpha}$ to (1.1),  the same proof of Theorem 1.1 yields that $u_{\alpha}$  converges smoothly to $u$ outside $\Sigma$.
    Then $u$ satisfies
    the harmonic map flow equation on $M\times [0,\infty)\setminus \Sigma$. The rest of the proof is exactly the same as in Theorem 7.2.3 of \cite{book}.
\end{proof}

Let $\mathcal S(u)$ be the singular set of the weak limit $u$;
i.e. for any $(x,t)\notin \mathcal S(u)$, there is $r>0$ such that
$u|_{B_r(x)\times (t-r^2,t+r^2)}$ is smooth. It follows from
Theorem \ref{thm:routine} that
\begin{equation*}
    \mathcal S(u) \subset \Sigma.
\end{equation*}
In fact, thanks to the local energy inequality (\ref{eqn:local}), we can say more about the position of $\mathcal S(u)$ in $\Sigma$. A point $z\in \Sigma$ is said to be a {\it rightmost} point of $\Sigma$ if for some $r>0$ we have
\begin{equation*}
    \Sigma \cap (B_r(x)\times (t-r^2,t)) =\emptyset.
\end{equation*}

\begin{lem}\label{thm:right}
    Every rightmost point of $\Sigma$ lies in $\mathcal S(u)$.
\end{lem}

\begin{proof}
    Let $z=(x,t)$ be a rightmost point of $\Sigma$. Assume that $z\notin \mathcal S(u)$. By definition, there exist $r_1>0$ and $r_2>0$ such that
    \begin{equation}
        \Sigma\cap (B_{r_1}(x)\times (t-r_1^2,t))=\emptyset
        \label{eqn:r1}
    \end{equation}
    and
    \begin{equation}
        \norm{\nabla u}_{C^0(B_{r_2}(x)\times (t-r_2,t+r_2))}\leq C.
        \label{eqn:r2}
    \end{equation}
    We may assume that $r_1>r_2$. There exists an $R\in (0,r_2)$ depending only on the $C$ in (\ref{eqn:r2}) and $\varepsilon_0$ such that
    \begin{equation}
        |B_R(x)| (1+4C^2)^2 \leq \varepsilon_0/8.
        \label{eqn:R}
    \end{equation}
    Given this $R$, we then find $\rho>0$ depending only on $R$, $\varepsilon_0$ and overall energy upperbound $E_0$ such that
    \begin{equation}
        \int_{M\setminus B_R(x)} (1+\abs{\nabla u_\alpha}^2)^\alpha e^{-\frac{\abs{x}^2}{4s}} \varphi^2 dv \leq \varepsilon_0/8
        \label{eqn:rho}
    \end{equation}
    for all $s\in (\rho^2,4\rho^2)$. Since
    \begin{equation*}
        (B_R(x)\times [t-4\rho^2,t-\rho^2] )\cap \Sigma=\emptyset,
    \end{equation*}
    by (i) of Theorem \ref{thm:routine}, we have for $k$ sufficiently large
    \begin{equation}
        \abs{\nabla u_{\alpha_k}}(y,s)\leq 2C
        \label{eqn:C2}
    \end{equation}
    for $y\in B_R(x)$ and $s\in [t-4\rho^2,t-\rho^2]$.
    Now we can estimate for $k$ large
    \begin{eqnarray*}
        \Psi^{\alpha_k}_\rho(u,(x,t)) &=& \rho^{2\alpha_k-2} \int_{T_\rho(t)} (1+\abs{\nabla u_\alpha}^2)^\alpha G_z \varphi^2 dv dt \\
        &\leq & 3\max_{s\in [t-4\rho^2,t-\rho^2]} \int_M (1+\abs{\nabla u_{\alpha_k}}^2)^{\alpha_k} e^{-\frac{\abs{x}^2}{4(t-s)}} \varphi^2 dv  \\
        &\leq &\frac 7 8 \varepsilon_0,
    \end{eqnarray*}
    where we used (\ref{eqn:R}), (\ref{eqn:rho}) and (\ref{eqn:C2}). This is a contradiction to the fact that $(x,t)\in \Sigma$.
\end{proof}

We now prove Theorem \ref{thm:compare}.
\begin{proof}
    Since the initial value is smooth, by uniqueness, $u(x,t)=\hat{u}(x,t)$ for $t<T$. (i) follows immediately from Lemma \ref{thm:right}.

    For (ii), since $u_\alpha$ converges smoothly to $u$ away from $\Sigma_T$, we have $\Sigma_T=\{p_i\}_{i=1}^k$. For $x\in M\setminus \Sigma_T$, by Theorem \ref{thm:routine}, we know $u$ is smooth on $B_r(x)\times [T-r^2,T+r^2]$ for some small $r>0$. This implies that $\nabla \hat{u}$ is bounded on $B_r(x)\times [T-r^2,T)$, which means that $x$ is not one of the $\hat{p}_i$'s. On the other hand, if $x$ is not a blow-up point for $\hat{u}$ at time $T$, then there exists $r>0$ such that
    \begin{equation}
        \abs{\nabla u}(x,t)\leq C
        \label{eqn:small}
    \end{equation}
    for $(x,t)$ in $B_r(x)\times [T-r^2,T)$.
    The same proof as in Lemma \ref{thm:right} shows that
    $x\notin \Sigma_T$. Hence, $\{p_i\}$ and $\{\hat{p}_i\}$ are the same set of points.
    Assume $k=1$. It remains to show $\hat{m}\leq m$. For $R>0$ and $\delta>0$, local energy inequality gives
    \begin{equation}\label{eqn:le}
        \int_{B_R(p)}e_\alpha(u_\alpha)(T) dv \leq \int_{B_{2R}} e_\alpha(u_\alpha) (T-\delta)  dv +CE_0\frac{\delta}{R^2}.
    \end{equation}
    For any $\varepsilon>0$, choose $R$ so that $\int_{B_{2R}(p)} e(\hat{u})(T)\leq \varepsilon/6$ and $\mbox{Vol}(B_{2R}(p))\leq \varepsilon/6$. Let $\delta>0$ be a small number such that $CE_0\frac{\delta}{R^2}\leq \varepsilon/6$. Taking $\alpha\to 1$ in (\ref{eqn:le}), we have
    \begin{equation*}
        m\leq \int_{B_{2R}(p)} e(\hat{u})(T-\delta) dv +\varepsilon/3.
    \end{equation*}
    Finally, letting $\delta$ go to zero, we obtain
    \begin{equation*}
        m\leq \int_{B_{2R}(p)}e(\hat{u})(T) dv +\hat{m} +\varepsilon/3\leq \hat{m} +\varepsilon/2.
    \end{equation*}
    The theorem follows by the arbitrariness of $\varepsilon$.
\end{proof}

With Theorem \ref{thm:routine} and Theorem \ref{thm:compare}, it
is natural to ask whether $u$ in Theorem \ref{thm:routine} is the
same as the Struwe  solution after the first blow-up time. Given
the nonuniqueness results \cite{unique1} and \cite{unique2}, one
can not exclude the possibility that $u$ is different from the
Struwe  solution. On this issue, we would like to make the
following remark,
\begin{rem}
    Consider maps from round $S^2$ to itself. According to \cite{CDY}, there exists an initial value map $u_0$ of degree three such that Struwe's solution blows up at some finite time $T$ and the homotopy class of the solution is changed for $t>T$. Hence the $\alpha-$flow solutions $u_\alpha$ with the same initial value can not converge strongly to Struwe's solution after $T$. A natural question is what we can say about $\Sigma$. See \cite{LT}.
\end{rem}

\section{An application}\label{sec:lin}
In this section, we apply our results about the $\alpha-$flow to the study of minimizing sequence of Dirichlet energy in a homotopy class. The following lemma is a variant of the main estimate in \cite{SU}.
    \begin{lem}\label{lem:regularity}
        Let $w$ be a map from $B_1$ to $N$ satisfying the following scaled equation for some $R>0$:
    \begin{equation}\label{eqn:su}
        \triangle w+ (\alpha-1)\frac{\nabla \abs{\nabla w}^2 \cdot \nabla w}{R^2+\abs{\nabla w}^2}+A(\nabla w,\nabla w)=h,
    \end{equation}
    where $h\in L^2$.
    There exists $\varepsilon_0>0$ such that if $E(w)<\varepsilon_0$ and $\alpha-1$ is sufficiently small then
    \begin{equation*}
        \norm{w-\bar{w}}_{W^{2,2}(B_{1/2})}\leq C( \norm{\nabla w}_{L^2(B_1)}+\norm{h}_{L^2(B_1)}).
    \end{equation*}
    Here $\bar{w}$ is the mean value of $w$ on $B_1$.
    \end{lem}
    \begin{proof}
        The proof is similar to the main estimate in \cite{SU}. During the proof, we write $\norm{\cdot}_{p,q}$ for $\norm{\cdot}_{W^{p,q}(B_1)}$. Multiplying the equation by $\varphi$ and taking the $L^p$ norm, we have
        \begin{equation}\label{eqn:lp}
            \norm{\triangle(\varphi w)}_{0,p}\leq 2(\alpha-1)\norm{\varphi w}_{2,p}+C \norm{ \abs{\nabla (\varphi w) }\abs{\nabla w}}_{0,p} +C \norm{w}_{1,p} +C\norm{h}_{0,p}.
        \end{equation}
        For $1<p<2$, the H\"older inequality implies that
        \begin{equation*}
            \norm{\abs{\nabla  (\varphi w)} \abs{\nabla w}}_{0,p}\leq \norm{ \nabla (\varphi w)}_{0,q} \norm{\nabla w}_{0,2},
        \end{equation*}
        where $q=2p/(2-p)$.
        By the $L^p$-estimate, for $\alpha-1$ small, we have
        \begin{equation*}
            \norm{\varphi w}_{2,p}\leq C\varepsilon_0 \norm{\nabla (\varphi w)}_{0,q}+C\norm{w}_{1,p} +C\norm{h}_{0,p}.
        \end{equation*}
        If we further assume that $\varepsilon_0$ is small, then it follows from the Sobolev embedding theorem
        \begin{equation*}
            \norm{\varphi w}_{2,p}\leq C\norm{w}_{1,p} +C\norm{h}_{0,p}.
        \end{equation*}
        Setting $p=4/3$ and using Sobolev embedding again, we have
        \begin{equation*}
            \norm{\varphi w}_{1,4}\leq C(\norm{w}_{1,4/3}+\norm{h}_{0,4/3}).
        \end{equation*}
        With this, we can apply  the interior $L^2$-estimate to (\ref{eqn:su}) and  multiply it  by $\varphi$ to get
        \begin{equation*}
        \norm{w-\bar{w}}_{W^{2,2}(B_{1/2})}\leq C(\norm{\nabla w}_{L^2(B_1)}+\norm{h}_{L^2(B_1)}).
        \end{equation*}
    \end{proof}

    We now prove Theorem \ref{thm:lin}.
\begin{proof}
    Since $u_i$ is smooth, we can find $\alpha_i>1$ such that
    \begin{equation}\label{eqn:alpha}
        E_{\alpha_i}(u_i)\leq E(u_i)+ \text {Vol}(M)+\frac{1}{i}.
    \end{equation}
    We then consider the $\alpha_i-$flow (\ref{eqn:flow1}) with initial value $u_i$ and denote the solution in Theorem \ref{thm:global} by $u_i(\cdot,t)$. Multiplying (\ref{eqn:flow1}) by $(1+\abs{\nabla u_i}^2)^{\alpha_i-1} \partial_t u_i$ and integrating by parts, we obtain a global energy inequality
    \begin{equation*}
        E_{\alpha_i}(u_i)=E_{\alpha_i}(u_i(\cdot,1))+\int_0^1 \int_M (1+\abs{\nabla u_i}^2)^{\alpha_i-1} \abs{\partial_t u_i}^2 dv dt.
    \end{equation*}
    By (\ref{eqn:alpha}),
    \begin{equation*}
        \lim_{i\to \infty} E_{\alpha_i}(u_i)=\lim_{i\to \infty} E(u_i)+ \text {Vol}(M).
    \end{equation*}
    Since $u_i$ is a minimizing sequence of $E$,
    \begin{equation*}
        \lim_{i\to \infty} E_{\alpha_i}(u_i(\cdot,1)) \geq \lim_{i\to \infty} E(u_i)+\text {Vol}(M).
    \end{equation*}
    Therefore,
    \begin{equation}\label{eqn:harmonic}
      \lim_{i\to \infty} \int_0^1\int_M (1+\abs{\nabla u_i}^2)^{\alpha_i-1}\abs{\partial_t u_i}^2 dv dt=0.
    \end{equation}
    Let $u$ be the weak limit of $u_{i}$ in $W^{1,2}(M\times [0,1])$ (by taking a subsequence if necessary), which is a weak solution to the harmonic map flow
    by Theorem \ref{thm:routine}. (\ref{eqn:harmonic}) implies that $u(\cdot,t)$ is (weak) harmonic map independent of $t$. Since $u_{i}(x,t)$ converges
    weakly to $u$ in $W^{1,2}(M\times [0,1])$  and the trace operator $T: W^{1,2}(M\times [0,1])\to L^2(M\times\{0\})$ is bounded linear operator,
    $u_i=u_{i}(\cdot,0)$ converges weakly  to $u(\cdot,0)$ in $L^2$. Hence, $u=u(\cdot,0)$ is a harmonic map.

    Instead of proving the energy identity for $u_i$ directly, we will find another sequence of maps $v_i$, which is also minimizing in the same homotopy class and satisfies some perturbed harmonic map equation (see (\ref{eqn:vi}) below).

    Thanks to (\ref{eqn:harmonic}), we may assume by taking subsequence if necessary
    \begin{equation}
        \int_{1/2}^1 \int_M \abs{\partial_t u_i}^2 dv dt\leq \frac{1}{16^i}.
        \label{eqn:thanks}
    \end{equation}
    Consider
    \begin{equation*}
        I_i=\left \{t\in [1/2,1]:\quad \int_M \abs{\partial_t u_i (\cdot, t)}^2 dv\geq \frac{1}{2^i}\right \}.
    \end{equation*}
    It follows from (\ref{eqn:thanks}) that
    \begin{equation*}
        \abs{I_i}\leq \frac{1}{8^i},
    \end{equation*}
    where $\abs{I_i}$ is the Lebesgue measure of $I_i$. Since
    \begin{equation*}
        \sum_{i=1}^{\infty} \abs{I_i}\leq \sum_{i=1}^\infty \frac{1}{8^i}\leq \frac{1}{4},
    \end{equation*}
    there exists at least one $t_0\in [1/2,1]$ such that for all
    $i$
    \begin{equation}\label{eqn:tension}
        \int_M \abs{\partial_t u_i(\cdot,t_0)}^2 dv \leq \frac{1}{2^i}.
    \end{equation}
    For simplicity, denote $u_i(x,t_0)$ by $v_i$ and write $h_i$ for $\partial_t u_i(x,t_0)$. Hence, $v_i$ satisfies the equation
    \begin{equation}\label{eqn:vi}
        \triangle v_i+ 2(\alpha_i-1)(\nabla ^2 v_i,\nabla v_i)\nabla v_i (1+\abs{\nabla v_i}^2)^{-1}+A(\nabla v_i,\nabla v_i)=h_i.
    \end{equation}
    As shown above, $v_i$ converges weakly  to $u$ in in $W^{1,2}$. Moreover, by (\ref{eqn:alpha}) and Lemma \ref{lem:decrease}, $v_i$ is also a minimizing sequence of $E$.
    With Lemma \ref{lem:regularity}, it is well known that there exists finitely many points $x_1,\cdots,x_k$ such that $v_i$ converges strongly to
    $u$ in $W^{1,2}$ away from these points. By the removable singularity theorem, $u$ can be extended to a smooth map on $M$. The proof of (\ref{eqn:identity})
    will be given in Section \ref{sec:identity}.

    Now, let us prove the second part of the theorem and assume that $\pi_2(N)$ is trivial. For simplicity, we assume $k=1$.
    Let $\eta$ be a smooth cutoff function which is $1$ for $r\geq 1$ and $0$ for $r\leq 1/2$. For some $\rho>0$, we define a new sequence of maps
     $\tilde{v}_i: M\to N$ such that $\tilde{v}_i$ is the same as $v_i$ outside $B_\rho(x_1)$ and for  $x\in B_\rho(x_1)$
    \begin{equation*}
        \tilde{v}_i(x)=\exp_{u(x)} \left( \eta(\frac{\abs{x}}{\rho})\exp^{-1}_{u(x)}\circ v_i(x) \right),
    \end{equation*}
    where $\exp$ is the exponential map on $N$. We claim that
    \begin{equation}\label{eqn:modify}
        \norm{\tilde{v}_i-u}_{W^{1,2}(M)}\to 0
    \end{equation}
    as $i\to \infty$. To see this, it suffices to consider $B_\rho(x_1)\setminus B_{\rho/2}(x_1)$ because $\tilde{v}_i\equiv u$ on $B_{\rho/2}(x_1)$
    and $\tilde{v}_i\equiv v_i$
    outside $B_\rho(x_1)$. On the other hand,  Lemma \ref{lem:regularity} implies that $v_i$ converges to $u$ on $B_{\rho}(x_1)\setminus B_{\rho/2}(x_1)$
    strongly in $W^{1,2}$ and $C^\beta$ for some $\beta>0$. Hence for large $i$, $v_i(B_{\rho(x_1)}\setminus B_{\rho/2}(x_1))$ lies in a small neighborhood of $u(x_1)$,
     where $\exp^{-1}_{u(x)}$ is a well defined smooth map (if $\rho$ is small). Since $F(y)=\exp_{u(x)}\left( \eta(\frac{\abs{x}}{\rho})\exp^{-1}_{u(x)} y \right)$ is
     a smooth map from a neighborhood of $u(x_1)$ into itself, we have
    \begin{eqnarray*}
        \norm{\tilde{v_i}-u}_{W^{1,2}(B_\rho\setminus B_{\rho/2} (x_1))} &=& \norm {F\circ {v_i}-F\circ u}_{W^{1,2}(B_{\rho}\setminus B_{\rho/2}(x_1))} \\
        &\leq& C \norm{v_i-u}_{W^{1,2}(B_\rho\setminus
        B_{\rho/2}(x_1))}\to 0 \quad \text {as } i\to\infty.
    \end{eqnarray*}
Thus the claim follows.
    Since $\pi_2(N)$ is trivial, $\tilde{v}_i$ is in the same homotopy class as $v_i$. Noticing that $v_i$ is a minimizing sequence of the Dirichlet energy and
    $v_i$ converges weakly to $u$ in $W^{1,2}$,  we have
    \begin{equation}\label{eqn:weak}
       E(u)\leq \lim_{i\to \infty} E(v_i)\leq \lim_{i\to \infty}
       E(\tilde{v}_i),
    \end{equation}
which implies
    \begin{equation*}
        E(u)=\lim_{i\to \infty} E(v_i).
    \end{equation*}
    Now, $v_i$ converges to $u$ strongly in $W^{1,2}$, which means that there is no energy concentration and Lemma \ref{lem:regularity} in turn shows
    that the convergence is in $C^\beta$ for some $\beta>0$.
\end{proof}

\section{An energy identity}\label{sec:identity}

Let $u_i$ be a sequence of maps from $M$ to $N$ with bounded energy such that

(1) an $\varepsilon-$regularity lemma such as Lemma \ref{lem:regularity} holds for each $u_i$;

(2) $u_i$ minimizes $E$ within some fixed homotopy class.

In the previous section, we have constructed a sequence of $v_i$ satisfying (1) and (2) by using $\alpha-$flow.
The purpose of this section is to prove that for such a sequence $u_i$, there exists $l$ harmonic maps $\omega_k(k=1,\cdots,l)$ from $\Real^2$ to $N$ such that
\begin{equation*}
    \lim_{i\to \infty} E(u_i) =E(u) +\sum_{k=1}^l E(\omega_k),
\end{equation*}
where $u$ is the weak limit of $u_i$ in $W^{1,2}$.

\subsection{Review of bubble tree construction}
Let us review the bubble tree construction due to Parker \cite{P}. After taking a subsequence which is still denoted by $u_i$,
we can decompose the domain into three parts. For each bubble points $x_k$, there is $c_{i,k}\in M$ and $\epsilon_{i,k},\lambda_{i,k}>0$ such that

(1) $u_i$ converges to $u$ on any compact subset of $M\setminus \{x_k\}$;

(2) if we rescale $u_i$ on $B_{i\lambda_{i,k}}(c_{i,k})$ to a ball of radius $i$, then there is again a weak limit $\omega_k$ (a level 1 bubble) and finitely many bubble points so that the decomposition happens again at a higher level.

(3) $A_{i,k}\stackrel{\triangle}{=}B_{\epsilon_{i,k}}(c_{i,k})\setminus B_{i\lambda_i,k}(c_{i,k})$ is the neck region. We denote $x_k$ by $p_k$ and $\omega_l(\infty)$ by $q_k$.

 \begin{figure}[h]
     \begin{center}
         \subfigure[bubble tree]{\label{fig:tree}\includegraphics[width=4cm]{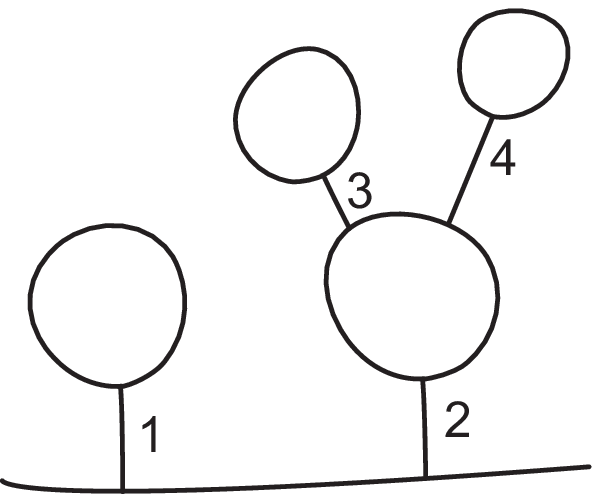}}
         \qquad \qquad
         \subfigure[domain decomposition]{\label{fig:shadow}\includegraphics[width=5cm]{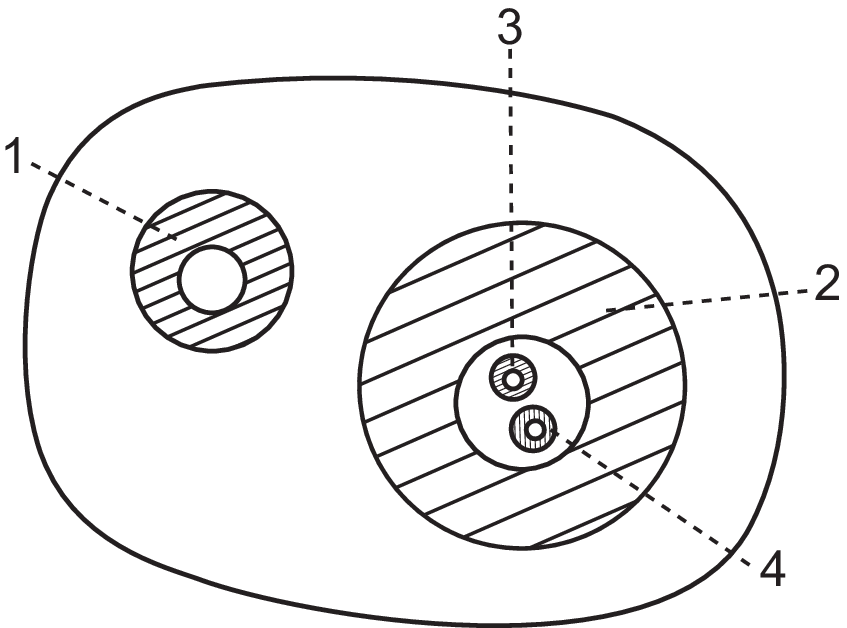}}
     \end{center}
     \caption{An example}
     \label{fig:example}
 \end{figure}

There are finitely many bubbles including the so-called ghost
bubbles which are constant maps and finitely many necks. For simplicity, we label bubbles
and necks by only one index. So $A_{i,k}$ could be a neck region
at any level. Suppose there are totally $L$ such necks. An example of buble tree and the decomposition are illustrated in Figure \ref{fig:example}.
The shadow parts in Figure \ref{fig:shadow} stand for the neck regions.

We will need the following fact, which was proved during the
construction of this bubble tree structure (see (1.6) in
\cite{P}):  On a neck region,
\begin{equation}\label{eqn:our}
  \int_{A_{i,k}} e(u_i) =C_R<\varepsilon_1/6,
\end{equation}
where $\varepsilon_1$ is the constant given in Lemma
\ref{lem:smallball}. To see this, recall that $\lambda_{i,k}$ is defined to be
\begin{equation*}
    \lambda_{i,k}=\mbox{largest }\lambda \mbox{ such that } \int_{B_{\epsilon_{i,k}(c_{i,k})}\setminus B_{\lambda}(c_{i,k})}e(u_i) \geq C_R
\end{equation*}
and $A_{i,k}\subset B_{\epsilon_{i,k}}(c_{i,k})\setminus B_{\lambda_{i,k}}(c_{i,k})$. By choosing $\delta$ small, $R$ large and a
new subsequence if necessary, we may require
\begin{equation}\label{eqn:smallneck}
    \int_{B_{\delta}\setminus B_{\lambda_{i,k}R}} e(u_i) <\varepsilon_1/3
\end{equation}
for $i$ large.

To prove (\ref{eqn:identity}), it suffices to show
\begin{equation*}
  \sum_{k=1}^L \lim_{i\to \infty} \int_{A_{i,k}} e(u_i)=0.
\end{equation*}

\subsection{An extension lemma}
  Let $p$ be a point in $N$ and $r$ be the injectivity radius of $N$. Consider a map $f:S^1\to N$ whose image lies in $B_r(p)$. Define $\tilde{f}:B_2\setminus B_1 \to N$ as
\begin{equation}
  \tilde{f}(\theta,r)=\exp_p [(2-r) \exp_p^{-1} f(\theta)].
  \label{eqn:closeup}
\end{equation}
It is obvious from the definition that $\tilde{f}(\theta,1)=f(\theta)$ and $\tilde{f}(\theta,2)=p$. We need to compute the energy of $\tilde{f}$. The chain rule implies
\begin{equation*}
  \abs{\pfrac{}{r} \tilde{f}(\theta,r)}=\abs{ D(\exp_p)_{(2-r)\exp_p^{-1} f(\theta)} \cdot \exp_p^{-1} f(\theta)} \leq Cd(f(\theta),p)
\end{equation*}
and
\begin{equation*}
  \abs{\pfrac{}{\theta} \tilde{f}(\theta,r)}= \abs{D(\exp_p)_{(2-r)\exp_p^{-1} f(\theta)} \cdot (2-r) D(\exp_p^{-1})_{f(\theta)} f'(\theta)} \leq C (2-r) f'(\theta).
\end{equation*}
Hence,
\begin{equation*}
    E(\tilde{f})\leq C  \int_{S^1} \int_1^2\left( d(f(\theta),p)^2+ \frac{1}{r^2} (2-r)^2 \abs{f'(\theta)}^2\right) r dr d\theta.
\end{equation*}
We summarize the computation in a lemma.
\begin{lem}
    \label{lem:cone}
    For any $\varepsilon>0$, there exists $\eta>0$ depending only on $\varepsilon$ and $N$ such that if $f:S^1\to N$ satisfies that
    \begin{equation*}
        \max_\theta [d(f(\theta),p) +\abs{f'(\theta)}] <\eta ,
    \end{equation*}
    then the Dirichlet energy of $\tilde{f}$ defined in (\ref{eqn:closeup}) on $B_2\setminus B_1$ is smaller than $\varepsilon$.
\end{lem}

\begin{rem}
  We also need a similar result for
  \begin{equation*}
    \tilde{f}(\theta,r)=\exp_p[(r-1)\exp_p^{-1}f(\theta)].
  \end{equation*}
  The proof is the same as that of Lemma \ref{lem:cone}, so we omit it.
\end{rem}
\subsection{A reference map}
We will construct a map from $M$ to $N$ whose image looks like the weak limit and bubbles connected by geodesics. It keeps the record of some topological information which will later be used to construct a new minimizing sequence.

\begin{lem} \label{lem:convexradius}
  There exists some positive constant $\sigma$ depending only on $N$ such that if $f_1,f_2$ are two continuous maps from $\Omega$ into $N$ such that
  \begin{equation*}
    d(f_1(x),f_2(x))\leq \sigma
  \end{equation*}
  for all $x\in \Omega$ and $f_1|_{\partial \Omega}=f_2|_{\partial \Omega}$, then there exists a homotopy deforming $f_1$ to $f_2$ with $\partial \Omega$ fixed.
\end{lem}

We omit the proof since it is obvious.

\begin{lem}\label{lem:smallball}
  There is a positive constant $\varepsilon_1$ depending only on $N$ such that any $W^{1,2}\cap C^0$ map $g:S^2\to N$ with $E(g)<\varepsilon_1$ is homotopic to a constant mapping.
\end{lem}
\begin{proof}
    Consider the harmonic map flow $g(t)$ starting from $g$. Since the energy is not enough for a bubble (if $\varepsilon_1$ is small) then we know the flow is smooth for all time and converges sequencially and strongly to a harmonic map $g_\infty$. In \cite{SU}, it is proved that the energy of such a harmonic map is no less than some constant depending only $N$ unless it is a constant map. By choosing $\varepsilon_1$ small, we may assume $g_\infty$ is a constant map. If $t_n\to \infty$ is a sequence of time such that $g(t_n)$ converges smoothly to $g_\infty$, we known from Lemma \ref{lem:convexradius} that $g(t_n)$ is homotopic to $g_\infty$ for $n$ large. Hence $g$ is homotopic to the constant map $g_\infty$.
\end{proof}

The next lemma shows how to squeeze a little room on both ends of a long neck region without changing the energy much.
\begin{lem}\label{lem:squeeze}
    For $i$ large so that $\delta >> \lambda_i R $, there is a diffeomorphism $f$ from $B_\delta\setminus B_{\lambda_i R}$ to $B_{\delta/4} \setminus B_{4\lambda_i R}$ such that for any map $u: B_\delta\setminus B_{\lambda_i R}\to N$, we have
    \begin{equation*}
        1-C(\frac{\delta}{\lambda_i R})\leq \frac{E(u\circ f^{-1}, B_{\delta/4}\setminus B_{4\lambda_i R})}{E(u,E_\delta\setminus E_{\lambda_i R})} \leq  1+C(\frac{\delta}{\lambda_i R}),
    \end{equation*}
    where $C(\frac{\delta}{\lambda_i R})$ is a constant which goes to zero if $\frac{\delta}{\lambda_i R}$ goes to infinity.
\end{lem}
\begin{proof}
  By conformal invariance of the Dirichlet energy, $u$ can be regarded as a map from $S^1\times [0,K]$ to $N$, where $K=\delta/(\lambda_i R)$.
  Define $\tilde{f}: S^1\times [0,K]\to S^1\times [4, K-4]$ by
  \begin{equation*}
    \tilde{f} (\theta,\rho) = (\theta, 4+\rho\frac{K-8}{K}).
  \end{equation*}
  It is straightforward to check
    \begin{equation*}
      1-C(K)\leq \frac{E(u\circ \tilde{f}^{-1}, S^1\times [4,K-4])}{E(u,S^1\times [0,K])} \leq  1+C(K),
    \end{equation*}
    for some $C(K)$ such that $C(K)$ goes to zero if $K$ goes to infinity. Set $s: S^1\times [0,K]\to B_\delta\setminus B_{\lambda_i R}$ to be
    \begin{equation*}
        s(\theta,\rho)=(\theta,\lambda_i R e^\rho).
    \end{equation*}
    Then $f: B_\delta\setminus B_{\lambda_i R}\to B_{\delta/4}\setminus B_{4\lambda_i R}$ defined by
    \begin{equation*}
        f=s\circ \tilde{f}\circ s^{-1}
    \end{equation*}
    is the diffeomorphism we need.
\end{proof}

A remark about the notation is needed for the rest of the proof. For simplicity, we omit the subscript for necks as if there was only one neck. In fact, one should repeat the construction or proof for each neck region. By the nature of the following proof, this should cause no further difficulty. We write $B_\delta\setminus B_{\lambda_i R}$ for a neck region, omitting the center. We assume the neck connects the weak limit and a level 1 bubble. For a general neck, one should replace $u$ with some bubble map $\omega_j$.

We need $\delta$, $R$ and $i$ to satisfy the following conditions.

{\bf {(A1)}} Let $\varepsilon=\min \{\varepsilon_1,\xi/L\}/8$,
where $\xi$ is a positive number introduced in the next subsection
and $L$ is the total number of necks. Lemma \ref{lem:cone} gives
an $\eta$. Choose $\delta$ and $R$ so that for each neck region
$B_\delta\setminus B_{\lambda_i R}$, $u$ restricted to $\partial
B_\delta$ satisfies
\begin{equation}\label{eqn:b1}
  \max_{\theta} \abs{d(u,p)} + \abs{\frac{d}{d\theta}u}\leq \min\{\eta/2,\sigma/2 \}
\end{equation}
where $p$ is the value of $u$ at the bubble point and $\omega$
restricted to $\partial B_R\subset\Bbb R^2$ satisfies
\begin{equation}\label{eqn:b2}
  \max_{\theta} \abs{d(\omega,q)} +\abs{\frac{d}{d\theta} \omega} \leq \min\{\eta/2,\sigma/2\},
\end{equation}
where $q$ is $\lim_{x\to \infty} \omega(x)$.
Fix $\delta$ and $R$, choose some $i$ so large that

{\bf {(A2)}} (\ref{eqn:b1}) and (\ref{eqn:b2}) remains true if we
replace $u$ by $u_i$, $\omega$ by $u_i(\lambda_i\cdot)$, $\eta/2$
by $\eta$ and $\sigma/2$ by $\sigma$;

{\bf {(A3)}} $\abs{u_i-u}\leq \sigma/3$ for $x\in M\setminus
B_\delta$ and $\abs{u_i(\lambda_i x)-\omega(x)}\leq \sigma/3$ for
all $x\in B_R$ where $\sigma$ is the constant in Lemma
\ref{lem:convexradius};

{\bf {(A4)}} Let $C(\cdot)$ be the function given in Lemma
\ref{lem:squeeze}. Let $i$ be so large that
$1+C(\frac{\delta}{\lambda_i R})<2$.

Now we define a new map $w:M\to N$ so that

(1) $w=u_i$ on $M\setminus B_\delta$ and $B_{\lambda_i R}$;

(2) $w=u_i\circ f^{-1}$ on $B_{\delta/4}\setminus B_{4\lambda_i R}$, where $f$ is the diffeomorphism constructed in Lemma \ref{lem:squeeze};

(3) $w(\theta,r)=\exp_p[\frac{r-\delta/2}{\delta/2}\exp_p^{-1} u_i (\theta,\delta)]$ for $(\theta,r)\in B_{\delta}\setminus B_{\delta/2}$;

(4) $w(\theta,r)=w(\theta,\frac{\delta^2}{4r})$ for $(\theta,r)\in B_{\delta/2}\setminus B_{\delta/4}$;

(5) $w(\theta,r)=\exp_q[\frac{2\lambda_i R-r}{\lambda_i R}\exp_q^{-1} u_i (\theta,\lambda_i R)]$ for $(\theta,r)\in B_{2\lambda_i R}\setminus B_{\lambda_i R}$;

(6) $w(\theta,r)=w(\theta,\frac{(2 \lambda_i R)^2}{r})$ for $(\theta,r)\in B_{4\lambda_i R}\setminus B_{2\lambda_i R}$.

It is easy to see that $w$ is homotopic to $u_i$. Now let us consider the map $w$ restricted to $B_{\delta/2}\setminus B_{2\lambda_i R}$. It follows from the construction that $w$ maps $\partial B_{\delta/2}$ to $p$ and $\partial B_{2\lambda_i R}$ to $q$. Hence, topologically it induces a map from $S^2$ to $N$ if we identify $\partial B_{2\lambda_i R}$ with the north pole $P_N$ and $\partial B_{\delta/2}$ with the south pole $P_S$.  This is in fact a homotopically trivial map. To see this, consider a map $w_1:\Real^2 \to N$ which agrees with $w$ on $B_{\delta/2}\setminus 2\lambda_i R$ and maps $B_{2\lambda_i R}$ to $q$ and $\Real^2\setminus B_{\delta/2}$ to $p$. By the conformal invariance of $E$, we regard $w_1$ as a map from $S^2$ to $N$. By (A1), (A2) (A4) and (\ref{eqn:smallneck}) , the energy of this map is smaller than $\varepsilon_1$ in Lemma \ref{lem:smallball}, hence it is homotopic to a constant map. Let $F:S^2\times [0,1]\to N$ be the homotopy such that $F(\cdot,0)=w_1(\cdot)$, $F(\cdot,1)\equiv p$ and $F(P_S,t)=p$ for any $t\in [0,1]$. The curve $F(P_N,\cdot)$ connects $q$ and $p$. Let $\gamma:[0,1]\to N$ be the shortest geodesic homotopic to this curve connecting $q$ and $p$.
\begin{figure}[h]
    \begin{center}
        \includegraphics[width=10cm]{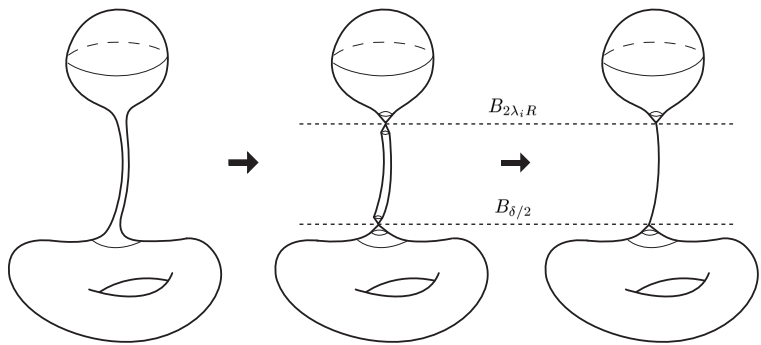}
    \end{center}
    \caption{reference map}
    \label{fig:ref}
\end{figure}

We can now define the reference map $\tilde{w}$. Let $\tilde{w}=w$ except that on $B_{\delta/2}\setminus B_{2\lambda_i R}$
\begin{equation*}
    \tilde w(\theta,r)=\gamma (\frac{\log r-\log (2\lambda_i R)}{\log (\delta/2)-\log (2\lambda_i R)}).
\end{equation*}
The above construction is illustrated by Figure \ref{fig:ref}.
We claim that $w$ and $\tilde{w}$ are homotopic. Since $w$ and $\tilde{w}$ are the same except on $B_{2\delta}\setminus B_{2\lambda_i R}$, we construct the homotopy explicitly. For $(\theta,r)\in B_{\delta/2}\setminus B_{2\lambda_i R}$,
\begin{equation*}
    G( (\theta,r),s)=
    \left\{
    \begin{array}[]{ll}
        F(P_N,\frac{r-2\lambda_i R}{\delta/2-2\lambda_i R}), & 2\lambda_i R\leq r\leq r(s) \\
        F( (\theta,(\delta/2-2\lambda_i R)\frac{r-r(s)}{\delta/2-r(s)}+2\lambda_i R),s), & r(s)\leq r\leq \delta/2,
    \end{array}
    \right.
\end{equation*}
where $r(s)=2\lambda_i R+s(\delta/2-2\lambda_i R)$.
From here, it is obvious that one can further deform $F(P_N,t)$ to $\gamma$.

If there are more than one neck region, we repeat the above
construction for each neck.

\subsection{Proof of the energy identity}

Let us assume that the energy identity (\ref{eqn:identity}) is not true. Then there
exists $\xi>0$ and a subsequence (still denoted by $u_i$) such
that
\begin{equation*}
    \lim_{i\to \infty} \sum_{k=1}^{L} \int_{A_{i,k}} e(u_i) \geq \xi.
\end{equation*}

Let $\tilde{w}$ be the reference map constructed above. We will
show a contradiction by construct a new sequence $\tilde{u}_i$ in
the same homotopy class such that
\begin{equation*}
    \lim_{i\to \infty} E(\tilde{u}_i)<\lim_{i\to \infty} E(u_i).
\end{equation*}

Define $\tilde{u}_i$ as follows

(1) $\tilde{u}_i=u_i$ outside the neck region ($B_\delta\setminus B_{\lambda_i R}$);

(2) $\tilde{u}_i (\theta,r)=\exp_p [\frac{r-\delta/2}{\delta/2}\exp_p^{-1} u_i(\theta,\delta)]$ for $(\theta,r)\in B_{\delta}\setminus B_{\delta/2}$;

(3) $\tilde{u}_i (\theta,r)=\exp_q [\frac{2\lambda_i R-r}{\lambda_i R}\exp_p^{-1} u_i(\theta,\lambda_i R)]$ for $(\theta,r)\in B_{2\lambda_i R}\setminus B_{\lambda_i R}$;

(4) $\tilde{u}_i (\theta,r)=\gamma (\frac{\log r-\log (2\lambda_i R)}{\log(\delta/2)-\log(2\lambda_i R)})$.

It remains to show first, the energy of $\tilde{u_i}$ in the neck
region is smaller than $\xi/2$ for $i$ large and second,
$\tilde{u}_i$ is homotopic to $\tilde{w}$. For the first point,
(A1) and Lemma \ref{lem:cone} implies that the energy in
$B_\delta\setminus B_{\delta/2}$ and $B_{2\lambda_i R}\setminus
B_{\lambda_i R}$ for all necks is no more than $\xi/4$. Moreover,
as $i$ gets larger, then the energy on the geodesic part can be
made as small as we want since the length of $\gamma$ is fixed
(See Section 5.3).

For the second point, if there is only one neck, then $M$ is
decomposed into three parts; $M\setminus B_{\delta/2}$,
$B_{\delta/2}\setminus B_{2\lambda_i R}$ and $B_{2\lambda_i R}$.
For the reference map $\tilde{w}$, the decomposition is fixed
because we chose and fixed a large $i$ in the construction. For
$\tilde{u_i}$, the center of $B_\delta$ and $B_{\lambda_i R}$
changes with $i$ and so does the radius of $B_{2\lambda_i R}$.
However, when restricted to the separating circles $\partial
B_{\delta/2}$ and $\partial B_{2\lambda_i R}$, $\tilde{w}$ and
$\tilde{u}_i$ are constant maps to the same points. This allows us
to prove the homotopy between $\tilde{w}$ and $\tilde{u}_i$ on
each of the three parts separately.

On $M\setminus B_{\delta/2}$, (A1) and (A3)
 show that $d(\tilde{w},\tilde{u}_i)<\sigma$, which implies that there is a homotopy between them leaving
 $\partial B_{\delta/2}$ fixed. The bubble part $B_{2\lambda_i R}$ of $\tilde{w}$ and $\tilde{u}_i$ are homotopic for the same reason.
 On $B_{\delta/2}\setminus B_{2\lambda_i R}$ part, they are different parametrization of the same curve.
 This concludes our proof for (\ref{eqn:identity}) in the case of only one neck.

 For the general case of multiple necks,
 the decomposition is more complicated but the same argument applies. To illustrate the idea,
 consider a bubble tree as shown in Figure \ref{fig:tree}. The necks are labelled from $1$ to $4$.
 The corresponding decomposition of $M$ is illustrated in Figure \ref{fig:shadow}.
The shadow parts are neck regions. There are eight circles
separating the neck regions and the bubble regions.
 Both $\tilde{w}$ and $\tilde{u}_i$ map these circles to $p_i$'s and $q_i$'s. The homotopy for bubble regions
 (including $M\setminus (B_{\delta/2}(x_1)\cup B_{\delta/2}(x_2))$ follows from Lemma \ref{lem:convexradius} and
  the homotopy for neck regions is obvious because they are different parametrization of the same geodesics.

   \section{Appendix}
  In this appendix, we prove the local existence and uniqueness of the $\alpha-$flow equation
  \begin{equation}\label{eqn:old}
    \partial_t u=\triangle_M u+(\alpha-1)\frac{\nabla \abs{\nabla u}^2 \cdot \nabla u}{1+\abs{\nabla u}^2} +A(u)(\nabla u,\nabla u)
\end{equation}
for any smooth initial value $u_0:M\to N\subset \Real^k$.

As in the case of harmonic maps, there is an intrinsic way of
writing the right hand side of (\ref{eqn:old}). If we regard
$(1+\abs{\nabla u}^2)^{\alpha-1} \nabla u$ as a section of the
pull back bundle $u^*TN$ and denote the induced connection of
$u^*TN$ by $\tilde{D}$, the righthand side can be written as
\begin{equation*}
    \tau_\alpha(u):=\frac{1}{(1+\abs{\nabla u}^2)^{\alpha-1}} \sum_i \tilde{D}_{e_i} \left( (1+\abs{\nabla u}^2)^{\alpha-1} \nabla_{e_i} u
    \right),
\end{equation*}
where $\{e_i\}$ is an orthonormal basis of $TM$ in some open set.
In the following, we will call $\tau_\alpha$ the $\alpha-$tension
field of $u$. It is obvious that smooth solution to
(\ref{eqn:old}) is equivalent to that of the following
\begin{equation*}
    \pfrac{u}{t}=\tau_\alpha(u),
\end{equation*}
where we regard $u(\cdot,t)$ as maps into $N$. This allows us to
consider a totally geodesic isometric embedding of $N$ into
$\Real^k$ with some non-flat metric as in \cite{Hamilton}.
Precisely, there is a metric $h$ on $\Real^k$ and an embedding $N$
into $\Real^k$ such that the isometric embedding is totally
geodesic and there exists an isometric involution $i$ which acts
on a tubular neighborhood $\mathcal T$ of $N$, leaving points of $N$ fixed.
By composing with this embedding, we regard $u$ as a map from $M$
into $(\Real^k,h)$ and denote by $\tau_\alpha^N$ and
$\tau_\alpha^{\Real^k}$ the $\alpha-$tension fields of $u$ as a
map into $N$ and into $\Real^k$ respectively. Since $N$ is a
totally geodesic submanifold, we have
$\tau_\alpha^N(u)=\tau_\alpha^{\Real^k}(u)$ (see page 108 in
\cite{Hamilton}). Therefore, it suffices to study
\begin{equation}
    \pfrac{u^\beta}{t}=\triangle_M u^\beta +(\alpha-1)\frac{\nabla \abs{\nabla u}^2\cdot \nabla u^\beta}{1+\abs{\nabla u}^2} +\Gamma(u)\# \nabla u\#\nabla u,
    \label{eqn:new}
\end{equation}
where $u^\beta$ is the components of $u$, $\triangle_M$ is the
usual Laplacian on $M$ and $\Gamma$ is the Christoffel symbol of
the Levi-Civita connection of $(\Real^k,h)$. Since $\abs{\nabla
u}^2=h_{\beta\gamma}(u)  \nabla {u^\beta}\cdot \nabla {u^\gamma}$,
when we expand $\nabla \abs{\nabla u}^2$, there is an extra term
involving the gradient of $h(u)$ which can be absorbed into the
last term. The equation can be rewritten as
\begin{equation}
    \pfrac{u^\beta}{t}=\triangle_M u^\beta +2(\alpha-1)\frac{\nabla^2_{ij}u^\gamma\nabla_i u^\gamma \nabla_j u^\beta }{1+\abs{\nabla u}^2}
    +\tilde{\Gamma}(u)\# \nabla u\# \nabla u.
    \label{eqn:latest}
\end{equation}

Next, we start a routine iteration procedure to show the local
existence of (\ref{eqn:latest}). Consider the following linear
parabolic system
\begin{equation}
    \pfrac{u}{t}=\triangle u + (\alpha-1) \frac{(\nabla^2 u, \nabla v)\nabla v}{1+\abs{\nabla v}^2}+ \tilde{\Gamma}(v)\#\nabla v\# \nabla v.
    \label{eqn:linear}
\end{equation}
Let $w=u-u_0$, then $w$ satisfies
\begin{eqnarray}\label{eqn:w}
    \pfrac{w}{t}&=& \triangle w + (\alpha-1) \frac{(\nabla^2 w, \nabla v)\nabla v}{1+\abs{\nabla v}^2}+ \tilde{\Gamma}(v)\#\nabla v\#\nabla v \\ \nonumber
    && +\triangle u_0 +(\alpha-1) \frac{(\nabla^2 u_0,\nabla v)\nabla v}{1+\abs{\nabla v}^2}.
\end{eqnarray}
For some $\mu\in (1,2)$, set
\begin{equation*}
    V=\left\{v\in C^{\mu,\mu/2}(M\times [0,T])|\quad \norm{v}_{C^{\mu,\mu/2}(M\times [0,T])}\leq 1, v(0)\equiv 0\right\}.
\end{equation*}
For each $v\in V$, the Schauder estimate of linear parabolic
systems (see \cite{Ed} and \cite{Schlag}) implies
\begin{equation}\label{eqn:schauder}
    \norm{w}_{C^{\mu+1,\frac{\mu+1}{2}}(M\times [0,T])}\leq C (\norm{w}_{C^0(M\times [0,T])}+  \norm{u_0}_{C^{\mu+1}(M)}+1).
\end{equation}
Since $w(\cdot,0)\equiv 0$, we have
\begin{equation*}
    \norm{w}_{C^0(M\times [0,T])}\leq T C (\norm{w}_{C^0(M\times [0,T])}+  \norm{u_0}_{C^{\mu+1}(M)}+1).
\end{equation*}
Choose $T$ small so that
\begin{equation*}
    \norm{w}_{C^0(M\times [0,T])}\leq TC(\norm{u_0}_{C^{\mu+1}(M)}+1).
\end{equation*}
By the interpolation of the H\"older space, for $\sigma>0$ such that
$\mu=(1-\sigma)(\mu+1)$, we have
\begin{equation*}
    \norm{w}_{C^{\mu,\mu/2}(M\times [0,T])}\leq \norm{w}_{C^{0}(M\times [0,T])}^\sigma \norm{w}_{C^{\mu+1,\frac{\mu+1}{2}}(M\times [0,T])}^{1-\sigma}.
\end{equation*}
This implies that we can choose $T$ sufficiently small so that
$w\in V$.

Now start from any $v_0\in V$. Let $v_{k+1}$ be the solution of
(\ref{eqn:linear}) with $v=v_{k}$ and zero initial value. We then
have a uniform bound
\begin{equation*}
    \norm{v_k}_{C^{\mu+1,\frac{\mu+1}{2}}(M\times [0,T])}\leq C.
\end{equation*}
Using Schauder estimate again, we obtain uniform
$C^{\mu+3,\frac{\mu+3}{2}}$ estimate for $v_k$. By taking a
subsequence, we know $v_k$ converges to some $w$ in
$C^{\mu+1,\frac{\mu+1}{2}}(M\times [0,T])$. Then $w+u_0$ is a
solution to the $\alpha-$flow with initial value $u_0$.

If there are two smooth solutions $u$ and $w$ to the equation
(\ref{eqn:latest}) defined on $[0,T]$, we can subtract the two
equations, multiply both sides by $u-w$ and integrate over $M$ to
get
\begin{eqnarray*}
    \frac{1}{2}\frac{d}{dt} \int_M \abs{w-u}^2 dv &\leq& -\int_M \abs{\nabla (u-w)}^2 dv \\
    && +2(\alpha-1) \int_M (\nabla^2_{ij}(u^\beta-w^\beta))\frac{ u^\beta_i u^\gamma_j}{1+\abs{\nabla u}^2} (u^\gamma-w^\gamma) dv \\
    && + C\int_M \abs{u-w}^2 dv +C\int_M \abs{\nabla u-\nabla w} \abs{u-w} dv.
\end{eqnarray*}
For the second term in the right hand side, integrating by parts
and noticing that $\frac{u^\beta_i u^\gamma_j}{1+\abs{\nabla u}^2}
\nabla_i (u^\beta-w^\beta) \nabla_j (u^\gamma-w^\gamma)\geq 0$
yields
\begin{equation*}
    \frac{d}{dt}\int_M \abs{u-w}^2 dv\leq C\int_M \abs{u-w}^2 dv,
\end{equation*}
which implies $u\equiv w$ on $M\times [0,T]$ if $u=w$ for $t=0$.
It follows immediately from the uniqueness and the involution
isometry $i:\mathcal T\to \mathcal T$ that if $u_0$ is a map from $M$ to $N$, the
image of the local solution $u(x,t)$ obtained above lies in $N$.

\begin{acknowledgement}{The research  of the first author was supported by the
Australian Research Council grant DP0985624. The second author was
supported  by a postdoctoral fellowship at the University of
Queensland through the grant DP0985624.}
\end{acknowledgement}


\begin{thebibliography}{}
  \bibitem{unique1}
    M. Bertsch, R. Dal Passo, R. Van der Hout:
    \newblock{Nonuniqueness for the heat flow of harmonic maps on the disk},
    \newblock{Arch. Ration. Mech. Anal.}, {\bf 161} (2002), no.2, 93-112.
  \bibitem{CDY}
    K.C. Chang, W.Y. Ding, R. Ye:
    \newblock{Finite-time blow-up of the heat flow of harmonic maps from surfaces},
    \newblock{J. Differential Geom.}, {\bf 36} (1992), no.2, 507-515.
  \bibitem{CT}
    J.Y. Chen, G. Tian:
    \newblock{Compactification of moduli space of harmonic mappings},
    \newblock{Comment. Math. Helv.}, {\bf 74} (1999), 201-237.
    \bibitem{CS}
        Y.M. Chen, M. Struwe:
        \newblock{Existence and partial regularity for heat flow for hamonic maps},
        \newblock{Math. Z.}, {\bf 201} (1989), 83-103.
\bibitem{DT}
    W.Y. Ding, G. Tian:
    \newblock{Energy identity for a class of approximate harmonic maps from surfaces},
    \newblock{Comm. Anal. Geom.}, {\bf 3} (1996), 543-554.
    \bibitem{DK}
        F. Duzaar, E. Kuwert:
        \newblock{Minimization of conformal invariant energies in homotopy classes},
        \newblock{Calc. Var. Partial Differential Equations}, {\bf 6} (1998), no.4, 285-313.
    \bibitem{ES}
        J. Eells, J.H. Sampson:
        \newblock{harmonic mappings of Riemannian manifolds,}
        \newblock{Amer. J. Math.}, {\bf 86} (1964) 109-160.
    \bibitem{Ed}
        S.D. Eidel'man:
        \newblock{Parabolic systems},
        \newblock{Translated from the Russian by Scripta Technica, London North-Holland Publishing Co.,} Amsterdam-London; Wolters-Noordhoff Publishing, Groningen 1969.
\bibitem{Hamilton}
    R.S. Hamilton:
    \newblock{Harmonic maps of manifolds with boundary},
    \newblock{Lecture Notes in Mathematics}, Vol. 471, Springer-Verlag, Berlin-New York, 1975.
\bibitem{Lamm}
    T. Lamm:
    \newblock{Energy identity for approximations of harmonic maps from surfaces},
    \newblock{Trans. Amer. Math. Soc.}, {\bf 362} (2010), 4077-4097.
      \bibitem{Lemaire}
    L. Lemaire:
    \newblock{Applications harmoniques des surfaces Riemanniennes},
    \newblock{J. Differential Geom.}, {\bf 13} (1978), 51-78.
    \bibitem{LT}
        J.Y. Li, G. Tian:
        \newblock{The blow-up locus of heat flows for harmonic maps},
        \newblock{Acta Math. Sin. (Engl. Ser.)}, {\bf 16} (2000), no.1, 29-62.
\bibitem{LiW}
    Y.X. Li, Y.D. Wang:
    \newblock{A weak energy identity and the length of necks for a sequence of Sacks-Uhlenbeck $\alpha$-harmonic maps},
    \newblock{Adv. Math.} (2010), doi:10.1016/j.aim.2010.03.020.
    \bibitem{Yuxiang}
        Y.X. Li, Y.D. Wang:
        \newblock{Bubbling location for sequences of approximate $f-$harmonic maps from surfaces};
        \newblock{Internat. J. Math.}, {\bf 21} (2010), no.4, 475-495.
    \bibitem{LW}
        F.H. Lin, C.Y. Wang:
        \newblock{Harmonic and quasi-harmonic spheres},
        \newblock{Comm. Anal. Geom.}, {\bf 7} (1999), no.2, 397-429.
    \bibitem{book}
        F.H. Lin, C.Y. Wang:
        \newblock{The analysis of harmonic maps and their heat flows},
        \newblock{World Scientific Publishing Co. Pte. Ltd. (2008).}
          \bibitem{P}
        T. Parker:
        \newblock{Bubble tree convergence for harmonic maps},
        \newblock{J. Differential Geom.}, {\bf 44} (1996), 595-633.
        \bibitem{SU}
         J. Sacks, K. Uhlenbeck:
        \newblock{The existence of minimal immersions of $2-$spheres},
    \newblock{Ann. of Math. (2)}, {\bf 113} (1981), no.1, 1-24.
    \bibitem{Schlag}
        W. Schlag:
        \newblock{Schauder and $L^p$ estimates for parabolic systems via companato spaces},
        \newblock{Comm. Partial Differential Equations}, {\bf 21} (1996), no. 7-8, 1141-1175.
\bibitem{Sch}
    R. Schoen:
    \newblock{Analytic aspects of the harmonic map problem},
    \newblock{Seminar on nonlinear partial differential equations (Berkeley, Calif., 1983)}, 321-358, Math. Sci. Res. Inst. Publ., 2, Springer, New York, 1984.
  \bibitem{SY}
    R. Schoen, S.T. Yau:
    \newblock{Existence of incompressible minimal surfaces and the topology of three dimensional manifolds with non-negative scalar curvature},
    \newblock{Ann. of Math.}, {\bf 110} (1979), 127-142.
\bibitem{struwe}
    M. Struwe:
    \newblock{On the evolution of harmonic maps of Riemannian surfaces},
    \newblock{Comm. Math. Helv.}, {\bf 60} (1985), 558-581.
\bibitem{mono}
    M. Struwe:
    \newblock{On the evolution of harmonic maps in higher dimensions},
    \newblock{J. Differential Geom.}, {\bf 28} (1988), 485-502.
      \bibitem{unique2}
    P. Topping:
    \newblock{Reverse bubbling and nonuniqueness in the harmonic map flow},
    \newblock{Int. Math. Res. Not.}, (2002) no.10, 505-520.
\end{thebibliography}
\end{document}